\documentclass{article}
\usepackage{amsfonts}
\usepackage{amsmath}
\usepackage{color}
\usepackage{amsmath,amssymb,bbm}
\usepackage{geometry}
\usepackage{graphicx}

\newtheorem{theorem}{Theorem}

\newtheorem{lemma}[theorem]{Lemma}

\newtheorem{problem}[theorem]{Problem}
\newtheorem{proposition}[theorem]{Proposition}
\newtheorem{remark}[theorem]{Remark}

\newenvironment{proof}[1][Proof]{\noindent\textbf{#1.} }{\ \rule{0.5em}{0.5em}}
\input{tcilatex}
\allowdisplaybreaks

\begin{document}

\title{The Stochastic Properties of $\ell ^{1}$-Regularized Spherical
Gaussian Fields}
\author{Valentina Cammarota and Domenico Marinucci \\
Department of Mathematics, University of Rome Tor Vergata}
\maketitle

\begin{abstract}
Convex regularization techniques are now widespread tools for solving
inverse problems in a variety of different frameworks. In some cases, the
functions to be reconstructed are naturally viewed as realizations from
random processes; an important question is thus whether such regularization
techniques preserve the properties of the underlying probability measures.
We focus here on a case which has produced a very lively debate in the
cosmological literature, namely Gaussian and isotropic spherical random
fields, and we prove that Gaussianity and isotropy are not conserved in
general under convex regularization over a Fourier dictionary, such as the
orthonormal system of spherical harmonics.\newline

\noindent \textbf{Keywords and Phrases}: Random Fields, Spherical Harmonics,
Gaussianity, Isotropy, Convex Regularization, Inpainting.\newline

\noindent \textbf{AMS Classification}: 60G60; 33C53, 43A90
\end{abstract}

\section{Introduction}

Let $T:M\rightarrow \mathbb{R}$ be a square integrable function on a
manifold $M$, and assume that the following is observed:%
\begin{equation}
T^{obs}:=\mathcal{A}T+n\text{ ,}  \label{invprob}
\end{equation}%
where $\mathcal{A}:L^{2}(M)\rightarrow $ $L^{2}(M)$ is a linear operator
that can represent, for instance, a blurring convolution or a mask setting
some values of the function $T$ to zero, while $n:M\rightarrow \mathbb{R}$
denotes observational noise. Recovering $T$ from observations on $T^{obs}$
is a standard example of a linear inverse problem, and it is now classical
to pursue a solution for (\ref{invprob}) by means of convex/$\ell ^{1}$%
-regularization procedures. More precisely, we can proceed by postulating
that the signal $T$ can be sparsely represented in a given dictionary $\Psi
, $ e.g., $T=\Psi \alpha _{0}$ where the vector $\alpha _{0}$ is assumed to
be sparse in a suitable sense, and then solving the $\ell ^{1}$-regularized
problem%
\begin{equation}
\alpha ^{reg}:=\arg \min_{\alpha }\left\{ \lambda \left\Vert \alpha
\right\Vert _{\ell ^{1}}+\frac{1}{2}\left\Vert T^{obs}-\mathcal{A}\Psi
\alpha \right\Vert _{L^{2}(S^{2})}^{2}\right\} \text{ ,}  \label{min1}
\end{equation}%
which can be viewed for instance as a form of Basis Pursuit Denoising \cite%
{chen} or a variation of the Lasso algorithm introduced in the statistical
literature by \cite{tibshirani}. Often the following alternative formulation
is considered:%
\begin{equation}
\alpha ^{reg}:=\arg \min_{\alpha }\left\{ \left\Vert \alpha \right\Vert
_{\ell ^{1}}\right\} \text{ subject to }\left\Vert T^{obs}-\mathcal{A}\Psi
\alpha \right\Vert _{L^{2}(S^{2})}\leq \varepsilon \text{ },  \label{min2}
\end{equation}%
for some $\varepsilon >0;$ it is known that there exist a bijection $\lambda
\leftrightarrow \varepsilon $ such that (\ref{min1}) and (\ref{min2}) have
the same solution \cite{StarckBook}. Many authors have worked on related
regularization problems over the last two decades - a very incomplete list
includes \cite{osborne}, \cite{daubechies}, \cite{efron}, \cite{fornasier},
\cite{mcewen}, \cite{wright}, see for instance \cite{StarckBook}, Chapter 7
for more references and a global overview. These results are also connected
to the rapidly growing literature on compressive sensing, see, e.g., \cite%
{donoho,baraniuk,Candes,rauhut,rauhut2}.

In many applied fields, it is customary to view $T$ as the realization of a
random field, and the reconstruction problems (\ref{min1}) and (\ref{min2})
are usually just the first steps before statistical data analysis (e.g.,
estimation and testing) is implemented. In other words, $T$ is viewed as a
random object on a probability space ($\Omega ,\Im ,P)$, $T(\omega
,x):=T:\Omega \times M\rightarrow \mathbb{R}$; hence it becomes important to
verify that $T^{reg}:=\Psi \alpha ^{reg}$, $T^{reg}:\Omega \times
M\rightarrow \mathbb{R}$, is close to $T$ in a meaningful probabilistic
sense. For instance, let $M$ be a homogeneous space of a compact group $%
\mathcal{G};$ a natural question is the following:

\begin{problem}
\label{problem1} Assume that the field $T$ is Gaussian and isotropic, e.g.,
the probability laws of $T(.)$ and $T^{g}(.)=T(g.)$ are the same for all $%
g\in \mathcal{G}.$ Is the random field $T^{reg}$ Gaussian and isotropic?
\end{problem}

The scenario we have described fits very well, for instance, the current
situation in the Cosmological literature, in particular in the field of
Cosmic Microwave Background (CMB) data analysis. The latter can be viewed as
a snapshot picture of the Universe at the so-called age of recombination,
e.g. $3.7\times 10^{5}$ years after the Big Bang (some 13 billion years
ago); its observation has been made possible by satellite experiments such
as WMAP \cite{WMAP} and Planck \cite{Planck}, which have raised an enormous
amount of theoretical and applied interest. CMB is usually viewed as a
single realization of a Gaussian isotropic random field on the sphere, e.g.,
$M=S^{2}$ and $\mathcal{G=}SO(3),$ the group of rotations in $\mathbb{R}^{3}$%
; observations are corrupted by observational noise and various forms of
convolutions (e.g., instrumental beams, masked regions) and a number of
efforts have been devoted to solving (\ref{invprob}) under these
circumstances. In this setting, algorithms such as (\ref{min1}) and (\ref%
{min2}) have been widely proposed, in some cases (see e.g., \cite{abrial1},
\cite{dupe,StarckBook,Starck1} and the references therein) taking as a
dictionary the orthonormal system of spherical harmonics $\left\{ Y_{\ell
m}\right\} .$ As well-known, the latter are eigenfunctions of the spherical
Laplacians $\Delta _{S^{2}}Y_{\ell m}=-\ell (\ell +1)Y_{\ell m}$ and lead to
the spectral representation%
\begin{equation*}
T(x)=\sum_{\ell =0}^{\infty }T_{\ell }(x)=\sum_{\ell =0}^{\infty
}\sum_{m=-\ell }^{\ell }a_{\ell m}Y_{\ell m}(x)\text{ .}
\end{equation*}%
Under Gaussianity and isotropy, this representation holds in the mean square
sense and the random coefficients are Gaussian and independent with variance
$Ea_{\ell m}\overline{a}_{\ell ^{\prime }m^{\prime }}=C_{\ell }\delta _{\ell
}^{\ell ^{\prime }}\delta _{m}^{m^{\prime }},$ the sequence $\left\{ C_{\ell
}\right\} $ representing the angular power spectrum (see for instance \cite%
{MarPecBook}). A very lively debate has then developed, to ascertain whether
in this setting the solution to the issue raised in Problem (\ref{problem1})
should allow for a positive or negative answer, see for instance \cite%
{Starck1}, \cite{StarckBayes} and the references therein. In particular, the
recent paper \cite{feeney} provides from an astrophysical perspective some
arguments and a large amount of numerical evidence to suggest that isotropy
will not hold in general.

The purpose of this paper is to address this question from a mathematical
point of view. To this aim, we will focus on idealistic circumstances where $%
\mathcal{A}$ is just the identity operator and noise $n$ is set identically
to zero, so that $T$ and $T^{obs}$ coincide. Of course, under these
circumstances the inverse problem would not really arise: however for our
aims these assumptions suffice, as we will show that even in this idealistic
setting stochastic properties such as Gaussianity and isotropy are not
preserved by regularization according to (\ref{min1}) or (\ref{min2}).

\subsection{Statement of the main results}

To establish our results, we shall first reformulate (\ref{min1}) and (\ref%
{min2}) in a form which is more directly amenable to stochastic analysis; in
particular, we shall show that:

\begin{proposition}
\label{equivalence} Let $T$ be a Gaussian isotropic spherical random field,
and denote by $\Psi $ the spherical harmonic dictionary. Then for any given $%
\delta ,\varepsilon >0,$ there exist a positive $\lambda =\lambda (\delta
,\varepsilon )$ such that the solution%
\begin{equation}
\alpha ^{reg}:=\arg \min_{\alpha }\left\{ \lambda (\delta ,\varepsilon
)\left\Vert \alpha \right\Vert _{\ell ^{1}}+\frac{1}{2}\left\Vert T-\Psi
\alpha \right\Vert _{L^{2}(S^{2})}^{2}\right\} \text{ }  \label{min4}
\end{equation}%
satisfies
\begin{equation}
\Pr \left\{ \left\Vert T-\Psi \alpha ^{reg}\right\Vert _{L^{2}(S^{2})}\leq
\varepsilon \right\} \geq 1-\delta \text{ }.  \label{min3}
\end{equation}
\end{proposition}

The previous result is stating that for a suitable choice of $\lambda $ the
solution to (\ref{min1}) satisfies the constraint in (\ref{min2}) with
probability arbitrarily close to one, so that the two problems can be seen
as substantially equivalent in a stochastic setting. Let us now write%
\begin{equation*}
T_{\delta ,\varepsilon }^{reg}(x):=\sum_{\ell m}a_{\ell m}^{reg}(\delta
,\varepsilon )Y_{\ell m}(x)=\sum_{\ell }T_{\ell ;\delta ,\varepsilon
}^{reg}(x)\text{ .}
\end{equation*}%
The main claim of this paper is the following

\begin{theorem}
\label{theo-aniso} The random fields $T_{\delta ,\varepsilon }^{reg}(.)$ are
necessarily anisotropic and nonGaussian, for any (arbitrarily small but
positive) values of $\delta ,\varepsilon .$
\end{theorem}

To make this claim more concrete, we shall also focus on the normalized
fourth-moment%
\begin{equation*}
\kappa _{\ell }(\theta ,\phi ):=\frac{E\{T_{\ell }^{reg}(\theta ,\phi )^{4}\}%
}{(E\{T_{\ell }^{reg}(0,0)^{2}\})^{2}}\text{ ,}
\end{equation*}%
which of course should be constant for all $(\theta ,\phi )\in S^{2}$ under
isotropy, and identically equal to 3 under Gaussianity. On the contrary, we
will provide an analytic expression for the value of $\kappa _{\ell }(\theta
,\phi )$ at the North Pole $N:(\theta ,\phi )=(0,0)$, as a function of the
angular power spectrum $C_{\ell }$ and the penalization parameter $\lambda
(\delta ,\varepsilon ).$ In particular, for the so-called complex-valued
regularization procedure (to be defined below), we shall show that

\begin{theorem}
\label{theo-trisp1} As $\lambda/\sqrt{C_\ell} \to \infty$, we have
\begin{equation*}
\lim_{\lambda/\sqrt{C_\ell} \to \infty} \frac{\log k_\ell(0,0)}{%
\lambda^2/C_{\ell}}=1.
\end{equation*}
\end{theorem}

Because the sequence $C_{\ell }$ is summable, the previous result entails
that the kurtosis of the field diverges exponentially at the North Pole as $%
\ell \rightarrow \infty $, showing an extremely nonGaussian behavior at high
frequencies. Under the same setting, we shall show that

\begin{theorem}
\label{theo-trisp2} As $\lambda /\sqrt{C_{\ell }}\rightarrow \infty ,$ we
have that%
\begin{equation*}
\lim_{\ell \rightarrow \infty }\frac{\kappa _{\ell }(\theta ,\phi )}{\kappa
_{\ell }(0,0)P_{\ell }^{4}(\cos \theta )}=1\text{ , for all }(\theta ,\phi
)\in S^{2}\text{ ,}
\end{equation*}%
$P_{\ell }(.)$ denoting the usual Legendre polynomial.
\end{theorem}

The latter result entails that the so-called trispectrum of the random field
is not constant over the sphere, as required by isotropy, but it rather
exhibits anisotropic oscillations. Under the so-called real-norm
regularization procedure (to be defined later), the asymptotic behavior is
slightly different, but anisotropy remains and the oscillations of the
trispectrum can again be predicted analytically, see below.

One heuristic intuition behind these results can be summarized as follows.
To understand the relationship between convex regularization and isotropy,
it can be convenient to view a problem like (\ref{min1}) as resulting from
the maximization of a Bayesian posterior distribution on the spherical
harmonic coefficients $a_{\ell m}$, assuming a Laplacian/Exponential prior
on these coefficients. We can now recall some earlier results from \cite%
{BaMa} (see also \cite{BMV,MarPecBook,BT}), showing that a random field
generated by sampling such independent non-Gaussian coefficients is
necessarily anisotropic; it can then be natural to conjecture that this
implicit anisotropy in the prior fields will persist in the regularized
maps. However, while this interpretation led us to conjecture the results of
this paper, it should be noted that it plays no role in the arguments that
follow. We refer again to \cite{feeney} for further discussion on these
issues and for a large set of numerical results.

The plan of the paper is as follows: in Section \ref{regularized}, we
discuss regularized estimates in a stochastic setting, and we establish
Proposition \ref{equivalence}; in Section \ref{anisotropy}, we prove that
regularized fields with the spherical harmonics dictionary are necessarily
anisotropic and nonGaussian, while in Section \ref{trispectra} the
trispectra and their asymptotic behavior are studied. Some final remarks are
collected in Section \ref{conclusions}.

\subsection{Acknowledgements}

We thank Stephen Feeney, Jason McEwen, Hiranya Peiris, Jean-Luc Starck and
Benjamin Wandelt for useful and lively discussions. This research is
supported by the European Research Council under the European Community
Seventh Framework Programme (FP7/2007-2013) ERC grant agreement no. 277742
\emph{Pascal}.

\section{$\ell ^{1}-$Regularized Random Fields \label{regularized}}

As motivated in the Introduction, we wish to consider the $\ell _{1}$
minimization problem
\begin{equation}
\left\{ a_{\ell m}^{reg}\right\} =\arg \min_{{\{a_{\ell m}\}}}\left\{
\lambda \sum_{\ell m}|a_{\ell m}|+\frac{1}{2}\left\Vert T^{obs}-\sum_{\ell
m}a_{\ell m}Y_{\ell m}\right\Vert _{L^{2}{(S^{2})}}^{{2}}\right\} \text{ ,}
\label{minprob0}
\end{equation}%
where as usual%
\begin{equation}
T^{obs}=\sum_{\ell m}a_{\ell m}^{obs}Y_{\ell m}\text{ .}  \label{specrap}
\end{equation}%
and {\ $\left\Vert .\right\Vert _{L^{2}(S^{2})}$ denotes the $L^{2}$ -norm
for functions on the sphere , e.g.,
\begin{equation*}
\left\Vert T^{obs}-\sum_{\ell m}a_{\ell m}Y_{\ell m}\right\Vert
_{L^{2}(S^{2})}^{2}=\int_{S^{2}}\left\vert T^{obs}-\sum_{\ell m}a_{\ell
m}Y_{\ell m}\right\vert ^{2}dx=\sum_{\ell m}\left\vert a_{\ell
m}^{obs}-a_{\ell m}\right\vert ^{2}\text{ .}
\end{equation*}%
} In equations (\ref{minprob0}) and (\ref{specrap}), we take as {usual} $%
\left\{ Y_{\ell m}\right\} $ to denote complex-valued spherical harmonics,
so that $Y_{\ell m}=(-1)^{m}\overline{Y}_{\ell {,-m}},$ the bar denoting
complex conjugations, and $\left\vert .\right\vert $ the complex modulus $%
|a_{\ell m}|:=\sqrt{[{\func{Re}}(a_{\ell m})]^{2}+[{\func{Im}}(a_{\ell
m})]^{2}};$ we label this case as the \emph{complex-valued regularization
scheme}. As an alternative, an orthonormal expansion into a real-valued
basis can be obtained by simply taking
\begin{equation}
T^{obs}=\sum_{\ell m}a_{\ell m}^{obs;\mathcal{R}}Y_{\ell m}^{\mathcal{R}}%
\text{ ,}  \label{specrap_real}
\end{equation}%
where {\ $a_{\ell 0}^{obs;\mathcal{R}}=a_{\ell 0}^{obs}$, $Y_{\ell 0}^{%
\mathcal{R}}=Y_{\ell 0}$, }
\begin{equation*}
a_{\ell m}^{obs;\mathcal{R}}=\sqrt{2}{\func{Re}}({a_{\ell m}^{obs}})\text{
for }{m>0}\text{ , }a_{\ell m}^{\mathcal{R}}={-}\sqrt{2}{\func{Im}}({a_{\ell
,-m}^{obs}})\text{ for }m<0\text{ ,}
\end{equation*}%
and%
\begin{equation*}
Y_{\ell m}^{\mathcal{R}}=\sqrt{2}{\func{Re}}(Y_{\ell m})\text{ for }{m>0}%
\text{ , }Y_{\ell m}^{\mathcal{R}}=\sqrt{2}{\func{Im}}(Y_{\ell {,-m}})\text{
for }m<0\text{ .}
\end{equation*}%
We are then led to the \emph{real-valued regularization scheme}%
\begin{equation}
\left\{ a_{\ell m}^{reg\ast }\right\} =\arg \min_{{\left\{ a_{\ell
m}\right\} }}\left\{ \lambda \sum_{\ell m}\left\vert a_{\ell m}^{\mathcal{R}%
}\right\vert +\frac{1}{2}\sum_{\ell m}\left\vert a_{\ell m}^{obs;\mathcal{R}%
}-a_{\ell m}^{\mathcal{R}}\right\vert ^{2}\right\} \text{ ,}
\label{minprob0_real}
\end{equation}%
where $\left\vert .\right\vert $ is standard absolute value for real
numbers. We shall consider both schemes in what follows.

The following two lemmas are standard, but nevertheless we report their
straightforward proofs for completeness. We shall use below the standard
polar coordinates for complex-valued random variables%
\begin{equation*}
a_{\ell m}=\rho _{\ell m}\exp (i\psi _{\ell m})\text{ ,}
\end{equation*}%
\begin{equation*}
\rho _{\ell m}:=\sqrt{[({\func{Re}}(a_{\ell m})]^{2}+[\func{Im}(a_{\ell
m})]^{2}}\text{ , }\psi _{\ell m}:=\arctan \frac{{\func{Re}}(a_{\ell m})}{{%
\func{Im}}(a_{\ell m})}\text{ .}
\end{equation*}%
Also, we denote by $\left\vert x\right\vert _{+}$ the positive part of the
real number $x.$

\begin{lemma}
If $T^{obs}$ is Gaussian and isotropic, we have that for $m\neq 0,$ $a_{\ell
m}^{obs}=\rho _{\ell m}^{obs}\exp (i\psi _{\ell m}^{obs}),$ where $\psi
_{\ell m}^{obs}\sim U[0,2\pi ]$ and the density of ${\rho _{\ell m}^{obs}}$
is given by%
\begin{equation*}
\Pr \left\{ {\rho _{\ell m}^{obs}} \leq R\right\} =\int_{0}^{R}f_{\rho ;\ell
}(r)dr\text{ ; }f_{\rho ;\ell }(r)=2\frac{r}{C_{\ell }}\exp\{-\frac{r^{2}}{%
C_{\ell }}\}\text{ .}
\end{equation*}
\end{lemma}

\begin{proof}
It suffices to notice that%
\begin{equation*}
\Pr \left\{ {\rho _{\ell m}^{obs}}\leq R\right\} =\Pr \left\{ {(\rho _{\ell
m}^{obs})}^{2}\leq R^{2}\right\}
\end{equation*}%
\begin{equation*}
=\Pr \left\{ \frac{1}{2}\frac{[{\func{Re}}({a_{\ell m}^{obs}})]^{2}+[{\func{%
Im}}({a_{\ell m}^{obs}})]^{2}}{C_{\ell }/2}\leq \frac{R^{2}}{C_{\ell }}%
\right\}
\end{equation*}%
\begin{equation*}
=\Pr \left\{ \frac{\chi _{2}^{2}}{2}\leq \frac{R^{2}}{C_{\ell }}\right\}
=1-\exp (\frac{R^{2}}{C_{\ell }})\text{ ,}
\end{equation*}%
{\ where $\chi _{2}^{2}$ is a Chi-squared with 2 degrees of freedom.} Whence
the result follows from differentiation.
\end{proof}

\begin{lemma}
The solution to (\ref{minprob0}) is provided by%
\begin{equation*}
a_{\ell m}^{reg}:={\func{Re}}(a_{\ell m}^{reg})+i{\func{Im}}(a_{\ell
m}^{reg})\text{ ,}
\end{equation*}%
where, for $\ell =1,...,\ell _{\max }$ and $m=-\ell ,...,\ell $ we have%
\begin{equation*}
{\func{Re}}(a_{\ell m}^{reg})=\left\vert \rho _{\ell m}^{obs}-\lambda
\right\vert _{+}\cos \psi _{\ell m}^{obs}\text{ , }{\func{Im}}(a_{\ell
m}^{reg})=\left\vert \rho _{\ell m}^{obs}-\lambda \right\vert _{+}\sin \psi
_{\ell m}^{obs}\text{ .}
\end{equation*}
\end{lemma}

\begin{proof}
We can rewrite (\ref{minprob0}) as
\begin{eqnarray}
\left\{ a_{\ell m}^{reg}\right\} &=&\arg \min_{{\left\{ a_{\ell m}\right\} }%
}\left\{ \lambda \sum_{\ell m}\left\vert a_{\ell m}\right\vert +\frac{1}{2}%
\left\Vert T^{obs}-\sum_{\ell m}a_{\ell m}Y_{\ell m}\right\Vert
_{L^{2}(S^{2})}^{2}\right\}  \notag \\
&=&\arg \min_{{\left\{ a_{\ell m}\right\} }}\left\{ \lambda \sum_{\ell
m}\left\vert a_{\ell m}\right\vert +\frac{1}{2}\sum_{\ell m}\left\vert
a_{\ell m}^{obs}-a_{\ell m}\right\vert ^{2}\right\} \text{.}
\label{jointinp}
\end{eqnarray}%
We can hence rewrite%
\begin{align*}
&\frac{1}{2}\sum_{\ell m}\left\vert a_{\ell m}^{obs}-a_{\ell m}\right\vert
^{2}+\lambda \sum_{\ell m}\left\vert a_{\ell m}\right\vert \\
&=\frac{1}{2}\sum_{\ell m}\left\{ [{\func{Re}}(a_{\ell m}^{obs})-{\func{Re}}%
(a_{\ell m})]^{2}+[{\func{Im}}(a_{\ell m}^{obs})-\func{Im}(a_{\ell
m})]^{2}\right\} +\lambda \sum_{\ell m}\left\vert a_{\ell m}\right\vert \\
&=\sum_{\ell m}v_{\ell m}\text{ ,}
\end{align*}%
where%
\begin{align*}
v_{\ell m}&=\frac{1}{2}(\rho _{\ell m}^{obs})^{2}\cos ^{2}\psi _{\ell
m}^{obs}+\frac{1}{2}\rho _{\ell m}^{2}\cos ^{2}\psi _{\ell m}-\rho _{\ell
m}^{obs}\rho _{\ell m}\cos \psi _{\ell m}^{obs}\cos \psi _{\ell m} \\
&+\frac{1}{2}(\rho _{\ell m}^{obs})^{2}\sin ^{2}\psi _{\ell m}^{obs}+\frac{1%
}{2}\rho _{\ell m}^{2}\sin ^{2}\psi _{\ell m}-\rho _{\ell m}^{obs}\rho
_{\ell m}\sin \psi _{\ell m}^{obs}\sin \psi _{\ell m}+\lambda \rho _{\ell m}
\\
&=\frac{1}{2}(\rho _{\ell m}^{obs})^{2}+\frac{1}{2}\rho _{\ell m}^{2}-\rho
_{\ell m}^{obs}\rho _{\ell m}\cos (\psi _{\ell m}^{obs}-\psi _{\ell
m})+\lambda \rho _{\ell m}\text{ .}
\end{align*}%
It is obvious that for any value of $\rho _{\ell m}^{obs},$ $v_{\ell m}$ is
minimized at $\psi _{\ell m}^{obs}=\psi _{\ell m};$ we are then led to the
following optimization problem:%
\begin{equation*}
\min_{{\{\rho _{\ell m}\}}}\sum_{\ell m}\phi (\rho _{\ell m}^{obs},\rho
_{\ell m};\lambda )\text{ ,}
\end{equation*}%
where%
\begin{equation*}
\phi (\rho _{\ell m}^{obs},\rho _{\ell m};\lambda )=\frac{1}{2}(\rho _{\ell
m}^{obs})^{2}+\frac{1}{2}\rho _{\ell m}^{2}-\rho _{\ell m}^{obs}\rho _{\ell
m}+\lambda \rho _{\ell m}\text{ .}
\end{equation*}%
Now it is standard calculus to show that, for $\lambda >\rho _{\ell m}^{obs}$%
\begin{equation*}
\frac{d\phi }{d\rho _{\ell m}}=\rho _{\ell m}+\lambda -\rho _{\ell m}^{obs}>0%
\text{ ,}
\end{equation*}%
while for $\lambda \leq \rho _{\ell m}^{obs}$%
\begin{equation*}
\frac{d\phi }{d\rho _{\ell m}}=\rho _{\ell m}+\lambda -\rho _{\ell
m}^{obs}=0\Longleftrightarrow \rho _{\ell m}=\rho _{\ell m}^{obs}-\lambda
\text{ .}
\end{equation*}%
The solution now follows immediately, given the global convexity of the
function $\phi (.).$
\end{proof}

\begin{remark}
The previous Lemma provides a simple generalization of the very well-known
fact that soft-thresholding provides the solution to (\ref{min1}) when the
dictionary is represented by an orthonormal basis of real valued functions.
In particular, for (\ref{minprob0_real}) the solution is immediately seen to
be given by
\begin{equation*}
{\ a_{\ell m}^{reg\ast }=sign(a_{\ell m}^{obs;\mathcal{R}})\left\vert
|a_{\ell m}^{obs;\mathcal{R}}|-\lambda \right\vert _{+}\text{ .} }
\end{equation*}%
It is important to note that the solution for the coefficient corresponding
to $m=0$ is exactly the same for both regularization schemes.
\end{remark}

The next result shows that, in the simplified circumstances we are
considering and for a {suitable} choice of the penalization parameter $%
\lambda $, the reconstruction error can be made arbitrarily small, with
probability arbitrarily close to one. For finite variance fields we have $%
E\{T^{2}\}=\sum_{\ell }\frac{2\ell +1}{4\pi }C_{\ell }<\infty $. To enforce
this condition and for notational convenience, in what follows we shall
assume that for all $\ell $, $0<C_{\ell }\leq K\ell ^{-\alpha }$,  for some $%
K>0$ and $\alpha >2$. This condition is minimal and fulfilled for instance
by all physically relevant models for CMB radiation.

\begin{proposition}
Under the above conditions, for all $\delta ,\varepsilon >0$ there exists {\
a positive} $\lambda =\lambda (\delta ,\varepsilon )$ such that
\begin{equation*}
T_{{\delta ,\varepsilon }}^{reg}(x)=\sum_{\ell m}a_{\ell m}^{reg}(\delta
,\varepsilon )Y_{\ell m}(x)
\end{equation*}%
\begin{equation*}
\left\{ a_{\ell m}^{reg}(\delta ,\varepsilon )\right\} =\arg \min_{\left\{
a_{\ell m}\right\} }\left\{ {\lambda (\delta ,\varepsilon )}\sum_{\ell
m}\left\vert a_{\ell m}\right\vert +\frac{1}{2}\sum_{\ell m}\left\vert
a_{\ell m}^{obs}-a_{\ell m}\right\vert ^{2}\right\} \text{ }
\end{equation*}%
and the solution satisfies%
\begin{equation*}
\Pr \left\{ \left\Vert T^{obs}-T_{{\delta ,\varepsilon }}^{reg}\right\Vert
_{L^{2}(S^{2})}<\varepsilon \right\} \geq 1-\delta \text{ .}
\end{equation*}%
The same result holds when the real-valued regularization scheme is adopted.
\end{proposition}

\begin{proof}
Note that
\begin{align*}
E\left\Vert T^{obs}-T^{reg}\right\Vert _{L^{2}(S^{2})}^{2}&=\sum_{\ell
m}E|a_{\ell m}^{obs}-a_{\ell m}^{reg}(\lambda )|^{2} \\
&=\sum_{\ell m}E{\{|a_{\ell m}^{obs}|^{2}\mathbb{I(}\left\vert a_{\ell
m}^{obs}\right\vert \leq \lambda )\}}+\lambda ^{2}\sum_{\ell m}E\mathbb{I(}%
\left\vert a_{\ell m}^{obs}\right\vert >\lambda )\text{ .}
\end{align*}%
Now fix $\ell ^{\ast }$ such that%
\begin{equation*}
\sum_{\ell {>}\ell ^{\ast }}(2\ell +1)C_{\ell }\leq \frac{\varepsilon }{4}%
\text{ ,}
\end{equation*}%
and note that%
\begin{equation*}
\sum_{\ell =1}^{\infty }\sum_{m}E{\{|a_{\ell m}^{obs}|^{2}\mathbb{I(}%
\left\vert a_{\ell m}^{obs}\right\vert \leq \lambda )\}}\leq \sum_{\ell
=1}^{\ell ^{\ast }}\sum_{m}E{\{|a_{\ell m}^{obs}|^{2}\mathbb{I(}\left\vert
a_{\ell m}^{obs}\right\vert \leq \lambda )\}}+\frac{\varepsilon }{4}
\end{equation*}%
where
\begin{eqnarray*}
\sum_{\ell =1}^{\ell ^{\ast }}\sum_{m}E{\{|a_{\ell m}^{obs}|^{2}\mathbb{I(}%
\left\vert a_{\ell m}^{obs}\right\vert \leq \lambda )\}} &\leq &\lambda
^{2}\sum_{\ell =1}^{\ell ^{\ast }}\sum_{m}E\{\mathbb{I(}\left\vert a_{\ell
m}^{obs}\right\vert {\leq }\lambda )\} \\
&{=}&{\lambda ^{2}\sum_{\ell =1}^{\ell ^{\ast }}\big[\Pr \{|a_{\ell
0}^{obs}|\leq \lambda \}+\sum_{m\neq 0}\Pr \{|a_{lm}^{obs}|\leq \lambda \}%
\big]} \\
&{=}&{\lambda ^{2}\sum_{\ell =1}^{\ell ^{\ast }}\big[\int_{-\lambda
}^{\lambda }\frac{1}{\sqrt{2\pi C_{\ell }}}\exp \{-\frac{u^{2}}{2C_{\ell }}%
\}du+\sum_{m\neq 0}\int_{0}^{\lambda }\frac{2u}{C_{\ell }}\exp \{-\frac{u^{2}%
}{C_{\ell }}\}du\big]} \\
&{=}&{\lambda ^{2}\sum_{\ell =1}^{\ell ^{\ast }}\big[\int_{-\lambda
}^{\lambda }\frac{1}{\sqrt{2\pi C_{\ell }}}\exp \{-\frac{u^{2}}{2C_{\ell }}%
\}du+\sum_{m\neq 0}\big(1-\exp \{-\frac{\lambda ^{2}}{C_{\ell }}\}\big)\big]}
\\
&{\leq }&{\lambda ^{2}\sum_{\ell =1}^{\ell ^{\ast }}\big[\int_{-\lambda
}^{\lambda }\frac{1}{\sqrt{2\pi C_{\ell }}}du+\sum_{m\neq 0}\frac{\sqrt{2}%
\lambda }{\sqrt{\pi C_{\ell }}}\big]} \\
&\leq &\lambda ^{2}\sum_{\ell =1}^{\ell ^{\ast }}\frac{(2\ell +1)\sqrt{2}%
\lambda }{\sqrt{\pi C_{\ell }}}\leq {\frac{\lambda ^{3}\sqrt{2}}{\sqrt{\pi {%
C_{\ell }}^{\ast }}}(2\ell ^{\ast }+{\ell ^{\ast }}^{2})}\leq \frac{%
\varepsilon }{4}\text{ , }
\end{eqnarray*}%
provided that $C_{\ell }^{\ast }:=\min_{\ell =1,..,\ell ^{\ast }}C_{\ell }$
and
\begin{equation*}
{\lambda ^{3}\leq {\frac{\varepsilon \sqrt{\pi C_{\ell }^{\ast }}}{4\sqrt{2}%
(2{\ell ^{\ast }}+{\ell ^{\ast }}^{2})}}\text{ .}}
\end{equation*}%
Let $\text{Erfc}$ be the complementary error function defined by $\text{Erfc}%
(x)=\frac{2}{\sqrt{\pi }}\int_{x}^{\infty }\exp \{-x^{2}\}dx$. Since for $%
x>0 $, $\text{Erfc}(x)$ is bounded by $\text{Erfc}(x)\leq \frac{2}{\sqrt{\pi
}}\frac{1}{x+\sqrt{x^{2}+\frac{4}{\pi }}}\exp \{-x^{2}\}\leq \exp \{-x^{2}\}$%
, to bound the second term, we note that
\begin{eqnarray*}
\lambda ^{2}\sum_{\ell m}E\{\mathbb{I(}\left\vert a_{\ell
m}^{obs}\right\vert >\lambda )\} &=&{\lambda ^{2}\sum_{\ell }\big[\Pr
\{|a_{\ell ,0}^{obs}|>\lambda \}+\sum_{m\neq 0}\Pr \{|a_{\ell
m}^{obs}|>\lambda \}\big]} \\
&=&{\lambda ^{2}\sum_{\ell }\big[2\int_{\lambda }^{\infty }\frac{1}{\sqrt{%
2\pi C_{\ell }}}\exp \{-\frac{u^{2}}{2C_{\ell }}\}du+\sum_{m\neq
0}\int_{\lambda }^{\infty }\frac{2u}{C_{\ell }}\exp \{-\frac{u^{2}}{C_{\ell }%
}\}du\big]} \\
&=&\lambda ^{2}\sum_{\ell }\big[\text{Erfc}(\frac{\lambda }{\sqrt{2C_{\ell }}%
})+\sum_{m\neq 0}\exp \{-\frac{\lambda ^{2}}{C_{\ell }}\}\big] \\
&=&\lambda ^{2}\sum_{\ell }(2\ell +1)\exp \{-\frac{\lambda ^{2}}{2C_{\ell }}%
\},
\end{eqnarray*}%
now, for a fixed $\ell ^{+}>1$, we write
\begin{equation*}
\lambda ^{2}\sum_{\ell m}E\{\mathbb{I(}\left\vert a_{\ell
m}^{obs}\right\vert >\lambda )\}\leq \lambda ^{2}(2\ell ^{+}+{\ell ^{+}}%
^{2})+\lambda ^{2}\sum_{\ell >\ell ^{+}}(2\ell +1)\exp \{-\frac{\lambda ^{2}%
}{2C_{\ell }}\}.
\end{equation*}%
Here we apply the integral test to the remainder of the series; since $%
f(\ell )=(2\ell +1)\exp \{-\frac{\lambda ^{2}}{2C_{\ell }}\}$ for $C_{\ell
}\le K \ell ^{-\alpha }$, $\alpha >2$, is a positive and monotonically
decreasing function for all $\ell \geq 1$, we have
\begin{align*}
\sum_{\ell >\ell ^{+}}(2\ell +1)\exp \{-\frac{\lambda ^{2}\ell ^{\alpha }}{2}%
\}& \leq \int_{\ell ^{+}}^{\infty }(2x+1)\exp \{-\frac{\lambda ^{2}x^{\alpha
}}{2}\}dx \\
& \leq \int_{0}^{\infty }(2x+1)\exp \{-\frac{\lambda ^{2}x^{2}}{2}\}dx=\frac{%
4+\lambda \sqrt{2\pi }}{2\lambda ^{2}}\text{ for all }\lambda \geq 0\text{ ,
}\alpha >2\text{ .}
\end{align*}%
Therefore
\begin{equation*}
\lambda ^{2}\sum_{\ell m}E\{\mathbb{I(}\left\vert a_{\ell
m}^{obs}\right\vert >\lambda )\}\leq \lambda ^{2}(2\ell ^{+}+{\ell ^{+}}%
^{2})+\lambda ^{2}\frac{4+\lambda \sqrt{2\pi }}{2\lambda ^{2}}\leq \frac{%
\varepsilon }{2}\text{ ,}
\end{equation*}%
provided that we take $\lambda $ such that
\begin{equation*}
\lambda \leq \min \left\{ \sqrt{\frac{\varepsilon }{4(2\ell ^{+}+{\ell ^{+}}%
^{2})}},\left( \frac{\varepsilon }{2}-4\right) \frac{1}{\sqrt{2\pi }},\sqrt[3%
]{{\frac{\varepsilon \sqrt{\pi C_{\ell }^{\ast }}}{4\sqrt{2}(2{\ell ^{\ast }}%
+{\ell ^{\ast }}^{2})}}}\right\} .
\end{equation*}%
The proof for the real-valued regularization scheme is entirely analogous.
\end{proof}

The previous result is straightforward, but it has some important consequenc{%
e}s for the interpretation of the results to follow in the next Sections. In
particular, it entails that the presence of nonGaussianity and anisotropy
after convex regularization is not due to poor approximation properties of
the reconstructed maps. The regularized fields we shall deal with can indeed
be viewed as solutions to the optimization problem: for $\delta ,\varepsilon
>0,$%
\begin{equation*}
\{a_{\ell m}^{reg}(\delta ,\varepsilon )\}:=\arg \min_{\{a_{\ell
m}\}}\left\{ \lambda (\delta ,\varepsilon )\sum_{\ell m}\left\vert a_{\ell
m}\right\vert +\sum_{\ell m}\left\vert a_{\ell m}^{obs}-a_{\ell
m}\right\vert^{2} \right\}
\end{equation*}%
where $\lambda (\delta ,\varepsilon )$ is such that%
\begin{equation*}
\Pr \left\{ \left\Vert T^{obs}-T^{reg}_{ \delta, \varepsilon} \right\Vert
_{L^{2}(S^{2})}>\varepsilon \right\} \leq \delta \text{ .}
\end{equation*}%
We shall show that even for $T$ Gaussian and $\delta ,\varepsilon $
arbitrary small (but positive), $T^{reg}_{ \delta, \varepsilon}$ exhibits
nonGaussian statistics which diverge to infinity at the highest frequencies.

\section{Anisotropy and NonGaussianity \label{anisotropy}}

Let us write as before%
\begin{equation*}
T^{reg}=\sum_{\ell }T_{\ell }^{reg}=\sum_{\ell m}a_{\ell m}^{reg}Y_{\ell m}%
\text{ , }T^{reg\ast }=\sum_{\ell }T_{\ell }^{reg\ast }=\sum_{\ell m}a_{\ell
m}^{reg\ast }Y_{\ell m}^{\mathcal{R}}\text{ ,}
\end{equation*}%
e.g., $T^{reg},T^{reg\ast }$ represent, respectively the $\ell ^{1}-$%
regularized maps under the complex and real-valued optimization schemes. For
the discussion to follow, we need to recall briefly the following result:

\begin{theorem}
\label{BaldiM} (See Ref. \cite{BaMa}) Assume the spherical harmonic
coefficients $\left\{ a_{\ell m}\right\} $ of an isotropic random field are
independent for $\ell =1,2,...$ and $m=0,1,...,\ell .$ Then they are
necessarily Gaussian.
\end{theorem}

This result was established in \cite{BaMa}, see also \cite{BMV}, \cite{BT}
for extensions to homogeneous spaces of more general compact groups and \cite%
{MarPecBook}, Theorem 6.12 for a proof. An obvious consequence is that a
sequence of independent, but nonGaussian, random coefficients $\left\{
a_{\ell m}\right\} $ will necessarily lead to anisotropic random fields.

We are now in the position to state and prove the first result of this
paper; here and in what follows, we use $\Phi (x)=(2\pi
)^{-1/2}\int_{-\infty }^{x}\exp \{-\frac{1}{2}u^{2}\}du$ to denote the
cumulative distribution function of a standard Gaussian variable.

\begin{theorem}
Let $T^{obs}$ be a Gaussian and isotropic spherical random field. Then the
fields $T^{reg},T^{reg\ast }$ are necessarily nonGaussian and anisotropic.
In particular, in the complex-valued regularization scheme we have
\begin{align*}
E\left\{ a_{\ell 0}^{reg}(\lambda )^{2}\right\}&=\gamma _{0}(\frac{\lambda }{%
\sqrt{C_{\ell }}}) :=C_{\ell }\left\{ (1+\frac{\lambda ^{2}}{C_{\ell }}%
)2(1-\Phi (\frac{\lambda }{\sqrt{C_{\ell }}}))-\exp (-\frac{\lambda ^{2}}{%
2C_{\ell }})\frac{\lambda }{\sqrt{C_{\ell }}}\sqrt{\frac{2}{\pi }}\right\}
\text{ ,}  \label{gamma0}
\end{align*}%
while for $m\neq 0$%
\begin{equation}
E\{\left\vert a_{\ell m}^{reg}(\lambda )\right\vert ^{2}\}=\gamma _{1}(\frac{%
\lambda }{\sqrt{C_{\ell }}}) :=C_{\ell }\left\{ \exp (-\frac{\lambda ^{2}}{%
C_{\ell }})-\frac{\lambda }{\sqrt{C_{\ell }}}\sqrt{\pi }2(1-\Phi (\frac{%
\sqrt{2}\lambda }{\sqrt{C_{\ell }}}))\right\} \text{ .}  \label{gamma1}
\end{equation}%
Moreover, for all $m\neq 0,$ we have that%
\begin{equation}
-\lim_{\lambda /\sqrt{C_{\ell }}\rightarrow \infty }\frac{2C_{\ell }}{%
\lambda ^{2}}\log \frac{E\{\left\vert a_{\ell m}^{reg}\right\vert ^{2}\}}{%
E\{\left\vert a_{\ell 0}^{reg}\right\vert ^{2}\}}=1\text{ .}  \label{ghione3}
\end{equation}
and%
\begin{equation*}
\lim_{\lambda /\sqrt{C_{\ell }}\rightarrow \infty }\frac{E\left\{ T_{\ell
}^{reg}(\theta ,\phi )^{2}\right\} }{E\left\{ T_{\ell }^{{reg}%
}(0,0)^{2}\right\} P_{\ell }^{2}(\cos \theta )}=1\text{ .}
\end{equation*}%
Finally, in the real-valued regularization scheme%
\begin{equation*}
E\left\{ a_{\ell 0}^{reg\ast }(\lambda )^{2}\right\} =E\{\left\vert a_{\ell
m}^{reg\ast }(\lambda )\right\vert ^{2}\}=\gamma _{0}(\frac{\lambda }{\sqrt{%
C_{\ell }}})\text{ ,}
\end{equation*}%
for all $m=-\ell ,...,\ell .$
\end{theorem}

\begin{proof}
By assumption, the input coefficients $\left\{ a_{\ell m}\right\} ,$ are
Gaussian and independent. The inpainted coefficients can be written $a_{\ell
m}^{reg}=j(a_{\ell m};\lambda ),$ where the function $j(.;\lambda )$ is
nonlinear; it follows immediately that they are independent and nonGaussian.
Hence the fields $T^{reg},T^{reg\ast }$ are necessarily anisotropic, in view
of Theorem \ref{BaldiM}. Focussing on $m=0,$ we have in particular%
\begin{equation*}
\Pr \left\{ a_{\ell 0}^{reg}=0\right\} =p_{\ell }(\lambda ):=\int_{-\lambda
}^{\lambda }\frac{1}{\sqrt{2\pi C_{\ell }}}\exp (-\frac{x^{2}}{2C_{\ell }}%
)dx>0
\end{equation*}%
so that the distribution of $a_{\ell 0}^{reg}=a_{\ell 0}^{reg}(\lambda )$ is
given by the mixture%
\begin{equation*}
p_{\ell }(\lambda )\delta _{0}+(1-\frac{p_{\ell }(\lambda )}{2})\Phi
^{+}(.;\lambda ,C_{\ell })+(1-\frac{p_{\ell }(\lambda )}{2})\Phi
^{-}(.;\lambda ,C_{\ell })\text{ ,}
\end{equation*}%
where $\Phi ^{+}(.;\lambda ,C_{\ell })$ is the distribution {\ of a}
Gaussian random variable with mean $-\lambda $ and conditioned to be positive%
$,$ and likewise $\Phi ^{-}(.;\lambda ,C_{\ell })$ is the distribution {\ of
a} Gaussian random variable with mean $\lambda $ and conditioned to be
negative. It is simple to see that we have%
\begin{equation*}
E\{a_{\ell 0}^{reg}(\lambda )\}=0\text{ ,}
\end{equation*}%
and%
\begin{eqnarray*}
E\left\{ a_{\ell 0}^{reg}(\lambda )^{2} \right\} &=&\frac{2}{\sqrt{2\pi
C_{\ell }}}\int_{0}^{\infty }y^{2}\exp \left\{ -\frac{(y+\lambda )^{2}}{%
2C_{\ell }}\right\} dy \\
&=&\frac{2}{\sqrt{2\pi C_{\ell }}}\int_{\lambda }^{\infty }(x-\lambda
)^{2}\exp \left\{ -\frac{x^{2}}{2C_{\ell }}\right\} dx \\
&=&\frac{2}{\sqrt{2\pi C_{\ell }}}\int_{\lambda }^{\infty }(x^{2}-2\lambda
x+\lambda ^{2})\exp \left\{ -\frac{x^{2}}{2C_{\ell }}\right\} dx \\
&=&\frac{2}{\sqrt{\pi }}C_{\ell }\int_{\lambda }^{\infty }\frac{x}{\sqrt{%
2C_{\ell }}}\exp \left\{ -\frac{x^{2}}{2C_{\ell }}\right\} d\frac{x^{2}}{%
2C_{\ell }}-\frac{4}{\sqrt{2\pi }}\sqrt{C_{\ell }}\int_{\lambda }^{\infty
}\lambda \exp \left\{ -\frac{x^{2}}{2C_{\ell }}\right\} d\frac{x^{2}}{%
2C_{\ell }} \\
&&+\frac{2}{\sqrt{\pi }}\lambda ^{2}\int_{\lambda }^{\infty }\exp \left\{ -%
\frac{x^{2}}{2C_{\ell }}\right\} d\frac{x}{\sqrt{{\ 2}C_{\ell }}} \\
&=&\frac{2C_{\ell }}{\sqrt{\pi }}\Gamma (\frac{3}{2};\frac{\lambda ^{2}}{%
2C_{\ell }})-\frac{4\sqrt{C_{\ell }}}{\sqrt{2\pi }}\lambda \exp \left\{ -%
\frac{\lambda ^{2}}{2C_{\ell }}\right\} +2\lambda ^{2}\left\{ 1-\Phi (\frac{%
\lambda }{\sqrt{C_{\ell }}})\right\} \\
&=&2C_{\ell }\left\{ \frac{1}{\sqrt{\pi }}\Gamma (\frac{3}{2};\frac{\lambda
^{2}}{2C_{\ell }})-\frac{2}{\sqrt{2\pi }}\frac{\lambda }{\sqrt{C_{\ell }}}%
\exp \left\{ -\frac{\lambda ^{2}}{2C_{\ell }}\right\} +\frac{\lambda ^{2}}{%
C_{\ell }}\left\{ 1-\Phi (\frac{\lambda }{\sqrt{C_{\ell }}})\right\}
\right\} \text{ ,}
\end{eqnarray*}%
where%
\begin{equation*}
\Gamma (p;c)=\int_{c}^{\infty }x^{p-1}\exp (-x)dx
\end{equation*}%
denotes the incomplete Gamma function. Now using
\begin{equation}
\Gamma (\frac{3}{2};c)=\sqrt{c}e^{-c}+\frac{1}{2}\sqrt{\pi }\; \text{Erfc}
\left( \sqrt{c}\right) \text{ , }\text{Erfc} (u):=2(1-\Phi (\sqrt{2}u))\text{
, }  \label{iniez1}
\end{equation}%
the previous expression can be further developed to obtain
\begin{align*}
E\left\{ a_{\ell 0}^{reg}(\lambda )^{2} \right\} &=2C_{\ell }\left\{ \frac{1%
}{\sqrt{\pi }} {\ \frac{\lambda}{\sqrt{2 C_\ell}}} \exp\{-\frac{\lambda ^{2}%
}{2C_{\ell }}\}+\frac{1}{2}\text{Erfc} \left( {\ \frac{\lambda}{\sqrt{2
C_\ell}}} \right) -\frac{2}{\sqrt{2\pi }}\frac{\lambda }{\sqrt{C_{\ell }}}%
\exp \left\{ -\frac{\lambda ^{2}}{2C_{\ell }}\right\} +\frac{\lambda ^{2}}{%
C_{\ell }}\left\{ 1-\Phi (\frac{\lambda }{\sqrt{C_{\ell }}})\right\} \right\}
\\
&=2C_{\ell }\left\{ 1-\Phi (\frac{\lambda }{\sqrt{C_{\ell }}})-\sqrt{\frac{1%
}{2\pi }}\frac{\lambda }{\sqrt{C_{\ell }}}\exp \left\{ -\frac{\lambda ^{2}}{%
2C_{\ell }}\right\} +\frac{{\lambda ^{2}}}{C_{\ell }}\left\{ 1-\Phi (\frac{%
\lambda }{\sqrt{C_{\ell }}})\right\} \right\} \\
&=C_{\ell }\left\{ (1+\frac{\lambda ^{2}}{C_{\ell }})2(1-\Phi (\frac{\lambda
}{\sqrt{C_{\ell }}}))-\exp (-\frac{\lambda ^{2}}{2C_{\ell }})\frac{\lambda }{%
\sqrt{C_{\ell }}}\sqrt{\frac{2}{\pi }}\right\} \text{ .}
\end{align*}
It can be checked easily that $\lim_{\lambda \rightarrow 0}E\left\{ a_{\ell
0}^{reg}(\lambda )^{2} \right\}=C_{\ell }$ , as expected. Similarly, by
moving to polar coordinates it is easy to see that we have%
\begin{eqnarray}
E\{\left\vert a_{\ell m}^{reg}(\lambda )\right\vert ^{2} \}&=&\frac{1}{2\pi }%
\int_{0}^{2\pi }\int_{\lambda }^{\infty }(r-\lambda )^{2}\frac{r}{C_{\ell }/2%
}\exp (-\frac{r^{2}}{C_{\ell }})drd\varphi  \notag \\
&=&C_{\ell }\int_{\lambda }^{\infty }(\frac{r-\lambda }{\sqrt{C_{\ell }}}%
)^{2}\frac{2r}{C_{\ell }}\exp (-\frac{r^{2}}{C_{\ell }})dr  \notag \\
&=&C_{\ell }\int_{\lambda }^{\infty }(\frac{r-\lambda }{\sqrt{C_{\ell }}}%
)^{2}\exp (-\frac{r^{2}}{C_{\ell }})d\frac{r^{2}}{C_{\ell }}  \notag \\
&=&C_{\ell }\left\{ \int_{\lambda ^{2}/C_{\ell }}^{\infty }u\exp (-u)du+%
\frac{\lambda ^{2}}{C_{\ell }}\int_{\lambda ^{2}/C_{\ell }}^{\infty }\exp
(-u)du\right\}  \notag \\
&&+C_{\ell }\left\{ -2\frac{\lambda }{\sqrt{C_{\ell }}}\int_{\lambda
^{2}/C_{\ell }}^{\infty }\sqrt{u}\exp (-u)du\right\} \text{ .}
\label{iniez3}
\end{eqnarray}%
Now using (\ref{iniez1}) and%
\begin{equation}
\Gamma (2;c)=\int_{c}^{\infty }u\exp (-u)du=e^{-c}+ce^{-c},  \label{iniez2}
\end{equation}%
we have%
\begin{eqnarray*}
E\{\left\vert a_{\ell m}^{reg}(\lambda )\right\vert ^{2} \}&=&C_{\ell
}\left\{ \exp (-\frac{\lambda ^{2}}{C_{\ell }})+\frac{\lambda ^{2}}{C_{\ell }%
}\exp (-\frac{\lambda ^{2}}{C_{\ell }}) +\frac{\lambda ^{2}}{C_{\ell }}\exp
(-\frac{\lambda ^{2}}{C_{\ell }}) \right\} \\
&&-2C_{\ell }\left\{ \frac{\lambda ^{2}}{C_{\ell }} \exp (-\frac{\lambda ^{2}%
}{C_{\ell }}) +\frac{1}{2}\sqrt{\pi }\text{Erfc} \left(\frac{\lambda}{\sqrt{%
C_\ell}} \right) \frac{\lambda }{\sqrt{C_{\ell }}}\right\} \\
&=&C_{\ell }\left\{ \exp (-\frac{\lambda ^{2}}{C_{\ell }})-\sqrt{\pi }\text{%
Erfc} \left( \frac{\lambda}{\sqrt{C_\ell}} \right) \frac{\lambda }{\sqrt{%
C_{\ell }}}\right\} \\
&=&C_{\ell }\left\{ \exp (-\frac{\lambda ^{2}}{C_{\ell }})-\frac{\lambda }{%
\sqrt{C_{\ell }}}\sqrt{\pi }2(1-\Phi (\frac{\sqrt{2}\lambda }{\sqrt{C_{\ell }%
}}))\right\} \text{ .}
\end{eqnarray*}%
Hence we have%
\begin{eqnarray*}
\frac{E\{\left\vert a_{\ell m}^{reg}\right\vert ^{2}\}}{E\{\left\vert
a_{\ell 0}^{reg}\right\vert ^{2}\}}&=&\frac{\int_{\lambda }^{\infty
}(r-\lambda )^{2}2\frac{r}{C_{\ell }}\exp (-\frac{r^{2}}{C_{\ell }})dr}{%
\frac{2}{\sqrt{2\pi C_{\ell }}}\int_{\lambda }^{\infty }(x-\lambda )^{2}\exp
\left\{ -\frac{x^{2}}{2C_{\ell }}\right\} dx} \\
&=&\frac{C_{\ell }\int_{\lambda }^{\infty }(\frac{r}{\sqrt{C_{\ell }}}-\frac{%
\lambda }{\sqrt{C_{\ell }}})^{2}2\frac{r}{\sqrt{C_{\ell }}}\exp (-\frac{r^{2}%
}{C_{\ell }})d\frac{r}{\sqrt{C_{\ell }}}}{\frac{2}{\sqrt{2\pi }}C_{\ell
}\int_{\lambda }^{\infty }(\frac{x}{\sqrt{C_{\ell }}}-\frac{\lambda }{\sqrt{%
C_{\ell }}})^{2}\exp \left\{ -\frac{x^{2}}{2C_{\ell }}\right\} d\frac{x}{%
\sqrt{C_{\ell }}}} \\
&=&\frac{\int_{\lambda /\sqrt{C_{\ell }}}^{\infty }(u-\frac{\lambda }{\sqrt{%
C_{\ell }}})^{2}u\exp (-u^{2})du}{\frac{1}{\sqrt{2\pi }}\int_{\lambda /\sqrt{%
C_{\ell }}}^{\infty }(u-\frac{\lambda }{\sqrt{C_{\ell }}})^{2}\exp \left\{ -%
\frac{u^{2}}{2}\right\} du} \\
&\leq &K_\varepsilon \exp (-\frac{\lambda ^{2}}{2C_{\ell }}(1-\varepsilon ))%
\frac{\int_{\lambda /\sqrt{C_{\ell }}}^{\infty }(u-\frac{\lambda }{\sqrt{%
C_{\ell }}})^{2}\exp (-\frac{u^{2}}{2})du}{\int_{\lambda /\sqrt{C_{\ell }}%
}^{\infty }(u-\frac{\lambda }{\sqrt{C_{\ell }}})^{2}\exp \left\{ -\frac{u^{2}%
}{2}\right\} du} \\
&=&K_{\varepsilon }\exp (-\frac{\lambda ^{2}}{2C_{\ell }}(1-\varepsilon ))%
\text{ },\text{ }
\end{eqnarray*}%
for some constant $K_{\varepsilon }>0,$ any $\varepsilon >0,$ because $u\exp
(-\frac{\varepsilon u^{2}}{2})\leq K_{\varepsilon }$ for all $u\geq \frac{%
\lambda }{\sqrt{C_{\ell }}}.$ Now%
\begin{align*}
&-\lim_{\lambda /\sqrt{C_{\ell }}\rightarrow \infty }\frac{2C_{\ell }}{%
\lambda ^{2}}\log \frac{E\{\left\vert a_{\ell m}^{reg}\right\vert ^{2}\}}{%
E\{\left\vert a_{\ell 0}^{reg}\right\vert ^{2}\}} \\
&=-\lim_{\lambda /\sqrt{C_{\ell }}\rightarrow \infty }\frac{2C_{\ell }}{%
\lambda ^{2}}\log K_{\varepsilon }-\lim_{\lambda /\sqrt{C_{\ell }}%
\rightarrow \infty }\frac{2C_{\ell }}{\lambda ^{2}}\log \exp (-\frac{\lambda
^{2}}{2C_{\ell }}(1-\varepsilon )) \\
&=(1-\varepsilon )\text{ ,}
\end{align*}%
and because $\varepsilon $ is arbitrary, the first result follows. To
conclude, it is then sufficient to note that%
\begin{equation*}
\frac{E\{T_{\ell }^{{reg}}(\theta ,\phi )^2 \}}{E\{T_{\ell }^{{reg}}(0,0)^2
\}P_{\ell }^{2}(\cos \theta )}=\frac{\sum_{m}E\{|a_{\ell
m}^{reg}|^{2}\}\left\vert Y_{\ell m}(\theta ,\phi )\right\vert ^{2}}{\frac{%
2\ell +1}{4\pi }E\{(a_{\ell 0}^{{reg}})^{2}\}P_{\ell }^{2}(\cos \theta )}
\end{equation*}%
\begin{equation*}
=1+\frac{\sum_{m\neq 0}E\{|a_{\ell m}^{reg}|^{2}\}\left\vert Y_{\ell
m}(\theta ,\phi )\right\vert ^{2}}{\frac{2\ell +1}{4\pi }E\{(a_{\ell 0}^{{reg%
}})^{2}\}P_{\ell }^{2}(\cos \theta )}\rightarrow 1\text{ ,}
\end{equation*}%
because%
\begin{equation*}
\frac{\sum_{m\neq 0}E\{|a_{\ell m}^{reg}|^{2}\} \left\vert Y_{\ell m}(\theta
,\phi )\right\vert ^{2}}{\frac{2\ell +1}{4\pi }E\{(a_{\ell 0}^{{reg}})^2
\}P_{\ell }^{2}(\cos \theta )}\leq {2\ell} \max_{m\neq 0}\frac{E\{|a_{\ell
m}^{reg}|^{2}\} \left\vert Y_{\ell m}(\theta ,\phi )\right\vert ^{2}}{\frac{%
2\ell +1}{4\pi } E\{(a_{\ell 0}^{{reg}})^{2}\} {P_{\ell}^2(\cos \theta)} }%
\rightarrow 0\text{ ,}
\end{equation*}%
as $\lambda /\sqrt{C_{\ell }}\rightarrow \infty$. The proof in the
real-valued scheme is analogous.
\end{proof}

\begin{remark}
As a consequence of the previous Theorem, in the complex-valued
regularization scheme the ratio $E\left\vert a_{\ell 0}^{reg}\right\vert
^{2}/E\left\vert a_{\ell m}^{reg}\right\vert ^{2}$ diverges to infinity
super-exponentially as $C_{\ell }\rightarrow 0,$ and the covariance function
is dominated by a single random coefficient, thus oscillating over the
sphere as the square of a Legendre polynomial. For the real-valued
algorithm, this is not the case, and the variance is constant; nevertheless,
this field is still anisotropic, as confirmed by the analysis of higher
order power spectra which we entertain in the next Sections.
\end{remark}

\section{High-Frequency Behavior of Trispectra \label{trispectra}}

\subsection{\protect\bigskip The Trispectrum at the North Pole}

A natural tool to explore non-Gaussian/anisotropic features of a spherical
random fields is provided by the expected values of higher-order moments of
its multipole components. For instance, fourth order moments lead to
so-called trispectra, see \cite{hu,MarPecBook,planckNG} for properties and
applications; for our aims, it suffices to recall that, under Gaussianity
and isotropy, we should have%
\begin{equation*}
\frac{ET_{\ell }^{4}(x)}{(ET_{\ell }^{2}(x))^{2}}\equiv 3\text{ , for all }%
x\in S^{2},\text{ }\ell =0,1,2,...
\end{equation*}%
In the following result we show instead that the normalized trispectrum of
convexly regularized fields diverges to infinity at the North Pole.

\begin{theorem}
The normalized trispectrum of a convexly regularized random field at the
North Pole is given by%
\begin{equation*}
\frac{E\{T_{\ell }^{reg} {(0,0)}^{4}\}}{[E\{T_{\ell }^{reg} {(0,0)}%
^{2}\}]^{2}}=\frac{E\left\{ a_{\ell 0}^{reg}(\lambda )^{4} \right\} }{[E\{
a_{\ell 0}^{reg}(\lambda )^{2} \} ]^{2}}={\ \sqrt{\frac{\pi}{2}}} \psi (%
\frac{\lambda }{\sqrt{C_{\ell }}})\text{ ,}
\end{equation*}%
where%
\begin{equation*}
\psi (\frac{\lambda }{\sqrt{C_{\ell }}}):=\frac{\int_{0}^{\infty }v^{4}\exp
\left\{ -\frac{1}{2}\left[ v+\frac{\lambda }{\sqrt{C_{\ell }}}\right]
^{2}\right\} dv}{\left[ \int_{0}^{\infty }v^{2}\exp \left\{ -\frac{1}{2}%
\left[ v+\frac{\lambda }{\sqrt{C_{\ell }}}\right] ^{2}\right\} dv\right] ^{2}%
}\text{ .}
\end{equation*}%
The function $\psi (.)$ is strictly positive and increasing, with%
\begin{equation*}
\lim_{x\rightarrow 0}\psi (x)=3\text{ , }\lim_{x\rightarrow \infty }\left[
\frac{15}{4} x^3 \exp\left\{\frac{x^2}{2}\right\} \right]^{-1}\psi(x)=1
\text{ .}
\end{equation*}
\end{theorem}

\begin{proof}
Similarly to the proof of the previous Theorem and using the same notation,
we have%
\begin{align*}
E\left\{ a_{\ell 0}^{reg}(\lambda )^{4} \right\} &=\frac{2}{\sqrt{2\pi
C_{\ell }}}\int_{\lambda }^{\infty }(x-\lambda )^{4}\exp \left\{ -\frac{x^{2}%
}{2C_{\ell }}\right\} dx \\
&=\frac{2}{\sqrt{2\pi C_{\ell }}}\int_{\lambda }^{\infty }x^{4}\exp \left\{ -%
\frac{x^{2}}{2C_{\ell }}\right\} dx-8\frac{1}{\sqrt{2\pi C_{\ell }}}%
\int_{\lambda }^{\infty }x^{3}\lambda \exp \left\{ -\frac{x^{2}}{2C_{\ell }}%
\right\} dx \\
&\;\;+12\frac{1}{\sqrt{2\pi C_{\ell }}}\int_{\lambda }^{\infty }x^{2}\lambda
^{2}\exp \left\{ -\frac{x^{2}}{2C_{\ell }}\right\} dx-8\frac{1}{\sqrt{2\pi
C_{\ell }}}\int_{\lambda }^{\infty }x\lambda ^{3}\exp \left\{ -\frac{x^{2}}{%
2C_{\ell }}\right\} dx \\
&\;\;+\frac{2}{\sqrt{2\pi C_{\ell }}}\int_{\lambda }^{\infty }\lambda
^{4}\exp \left\{ -\frac{x^{2}}{2C_{\ell }}\right\} dx \\
&=\frac{2}{\sqrt{2\pi }}\sqrt{8}C_{\ell }^{2}\int_{\lambda }^{\infty }\frac{%
x^{3}}{\sqrt{8C_{\ell }^{3}}}\exp \left\{ -\frac{x^{2}}{2C_{\ell }}\right\} d%
\frac{x^{2}}{2C_{\ell }} \\
&\;\;-8\frac{1}{\sqrt{2\pi }}2\sqrt{C_{\ell }^{3}}\lambda \int_{\lambda
}^{\infty }(\frac{x}{\sqrt{2C_{\ell }}})^{2}\exp \left\{ -\frac{x^{2}}{%
2C_{\ell }}\right\} d\frac{x^{2}}{2C_{\ell }} \\
&\;\;+12\frac{1}{\sqrt{2\pi }}\sqrt{2}C_{\ell }\lambda ^{2}\int_{\lambda
}^{\infty }\frac{x}{\sqrt{2C_{\ell }}}\exp \left\{ -\frac{x^{2}}{2C_{\ell }}%
\right\} d\frac{x^{2}}{2C_{\ell }} \\
&\;\;-8\frac{{\ \sqrt{C_\ell}}}{\sqrt{2\pi }}\lambda ^{3}\int_{\lambda
}^{\infty }\exp \left\{ -\frac{x^{2}}{2C_{\ell }}\right\} d\frac{x^{2}}{%
2C_{\ell }} +\frac{2}{\sqrt{\pi }}\lambda ^{4}\int_{\lambda }^{\infty }\exp
\left\{ -\frac{x^{2}}{2C_{\ell }}\right\} d\frac{x}{\sqrt{2C_{\ell }}} \\
&=2C_{\ell }^{2}\frac{\sqrt{8}}{\sqrt{2\pi }}\left\{ \Gamma (\frac{5}{2};%
\frac{\lambda ^{2}}{2C_{\ell }})-\frac{4}{\sqrt{2}}\frac{{\ \sqrt{2}}\lambda
}{\sqrt{2C_{\ell }}}\Gamma (2;\frac{\lambda ^{2}}{2C_{\ell }})+\frac{6}{%
\sqrt{2^{2}}}(\frac{\lambda }{\sqrt{{C_{\ell }}}})^{2}\Gamma (\frac{3}{2};%
\frac{\lambda ^{2}}{2C_{\ell }})\right. \\
&\;\;\left. - {\ \frac{\sqrt{2} \lambda^3}{C_\ell \sqrt{C_\ell}} \exp\left\{-%
\frac{\lambda^2}{2 C_\ell}\right\} }+\frac{4 {\ \sqrt{2 \pi}} }{\sqrt{2^{3}}}%
(\frac{\lambda }{\sqrt{2C_{\ell }}})^{4}(1-\Phi (\frac{\lambda }{\sqrt{{\ 2
C_{\ell }}}}))\right\}.
\end{align*}
Observing that $\Gamma(1;c)=e^{-c}$, we obtain
\begin{align*}
&E\left\{ a_{\ell 0}^{reg}(\lambda )^{4} \right\} \\
&=C_{\ell }^{2}\frac{4}{\sqrt{\pi }} \left\{ \sum_{k=2}^{5} (-1)^{k+1} \left(%
\frac{\nu _{\ell }}{\sqrt{2}}\right)^{5-k} \binom{4}{5-k} \Gamma \left(\frac{%
k}{2};\left(\frac{\nu _{\ell }}{\sqrt{2}}\right)^2 \right)+ 2 \sqrt \pi
\left(\frac{\nu _{\ell }}{\sqrt{2}}\right)^{4} (1-\Phi (\frac{\nu _{\ell }}{%
\sqrt{2}}) \right\} \text{ ,}
\end{align*}%
where $\nu _{\ell }:=\frac{\lambda }{\sqrt{C_{\ell }}}.$ Again, it is simple
to check that
\begin{eqnarray*}
\lim_{\lambda/ \sqrt{C_\ell} \rightarrow 0}E\left\{ a_{\ell 0}^{reg}(\lambda
)^{4} \right\} &=&2C_{\ell }^{2}\frac{\sqrt{8}}{\sqrt{2\pi }}\lim_{\lambda/
\sqrt{C_\ell} \rightarrow 0}\Gamma (\frac{5}{2};\frac{\lambda ^{2}}{2C_{\ell
}}) \\
&=&2C_{\ell }^{2}\frac{\sqrt{8}}{\sqrt{2\pi }}\frac{3}{4}\sqrt{\pi }%
=3C_{\ell }^{2}\text{ ,}
\end{eqnarray*}%
as expected, because $\lim_{\lambda/ \sqrt{C_\ell} \rightarrow 0}E\left\{
a_{\ell 0}^{reg}(\lambda )^{4} \right\} /[E\left\{ a_{\ell 0}^{reg}(\lambda
)^{2} \right\} ]^{2}=3$ provides the fourth moment of a standard Gaussian
variable. Note also that%
\begin{eqnarray*}
\psi(\frac{\lambda}{\sqrt{C_\ell}})&=&\frac{\int_{0}^{\infty }v^{4}\exp
\left\{ -\frac{1}{2}\left[ v+\frac{\lambda }{\sqrt{C_{\ell }}}\right]
^{2}\right\} dv}{ \left[ \int_{0}^{\infty }v^{2}\exp \left\{ -\frac{1}{2}%
\left[ v+\frac{\lambda }{\sqrt{C_{\ell }}}\right] ^{2}\right\} dv\right] ^{2}%
} \\
&=&\exp \left\{ \frac{1}{2}\frac{\lambda^2 }{C_{\ell } } \right\} \frac{%
\int_{0}^{\infty }v^{4}\exp \left\{ -\frac{1}{2}\left[ v^{2}+\frac{2\lambda {%
v} }{\sqrt{C_{\ell }}}\right] \right\} dv}{\left[ \int_{0}^{\infty
}v^{2}\exp \left\{ -\frac{1}{2}\left[ v^{2}+\frac{2\lambda {v} }{\sqrt{%
C_{\ell }}}\right] \right\} dv\right] ^{2}} \\
&=&\exp \left\{ \frac{1}{2} \frac{\lambda^2 }{ C_{\ell } } \right\} \frac{-5
\frac{\lambda }{\sqrt{C_{\ell }}} - (\frac{\lambda }{\sqrt{C_{\ell }}}%
)^3+\exp\{\frac{\lambda^2 }{2 C_{\ell }}\} \sqrt{\frac{\pi}{2}}(3+ 6 \frac{%
\lambda^2 }{C_{\ell }} +\frac{\lambda^4 }{C_{\ell }^2} )\text{Erfc}(\frac{%
\lambda }{\sqrt{ 2 C_{\ell }}}) }{\left[- \frac{\lambda }{\sqrt{C_{\ell }}}
+\exp\{\frac{\lambda^2 }{2 C_{\ell }}\} \sqrt{\frac{\pi}{2}}(1+ \frac{%
\lambda^2 }{C_{\ell }} )\text{Erfc}(\frac{\lambda }{\sqrt{ 2 C_{\ell }}}) %
\right]^2} \text{ ,}
\end{eqnarray*}
here we use the classical asymptotic expansion of the complementary error
function, i.e., for large $x$ we have $\text{Erfc}(x)=\frac{e^{-x^2}}{x
\sqrt{\pi}} (1- \frac{1}{2 x^2}+ \frac{3}{4 x^4}+O(x^{-5}))$, then
\begin{eqnarray*}
&&\lim_{\lambda / \sqrt{C_\ell} \to \infty} \psi(\frac{\lambda}{\sqrt{C_\ell}%
}) \left[ \exp \left\{ \frac{1}{2}\frac{\lambda^2 }{{C_{\ell }}}\right\}
\frac{15}{4} \frac{\lambda^3}{C_\ell^{\frac 3 2}} \right]^{-1} \\
&=&\lim_{\lambda / \sqrt{C_\ell} \to \infty}\frac{-5 \frac{\lambda }{\sqrt{%
C_{\ell }}} - (\frac{\lambda }{\sqrt{C_{\ell }}})^3+\exp\{\frac{\lambda^2 }{%
2 C_{\ell }}\} \sqrt{\frac{\pi}{2}}(3+ 6 \frac{\lambda^2 }{C_{\ell }} +\frac{%
\lambda^4 }{C_{\ell }^2} )\text{Erfc}(\frac{\lambda }{\sqrt{ 2 C_{\ell }}})
}{ \frac{15}{4} \frac{\lambda^3}{C_\ell^{\frac 3 2}} \left[- \frac{\lambda }{%
\sqrt{C_{\ell }}} +\exp\{\frac{\lambda^2 }{2 C_{\ell }}\} \sqrt{\frac{\pi}{2}%
}(1+ \frac{\lambda^2 }{C_{\ell }} )\text{Erfc}(\frac{\lambda }{\sqrt{ 2
C_{\ell }}}) \right]^2} \\
&=&\lim_{\lambda / \sqrt{C_\ell} \to \infty} \frac{-5 \frac{\lambda }{\sqrt{%
C_{\ell }}} - (\frac{\lambda }{\sqrt{C_{\ell }}})^3+\exp\{\frac{\lambda^2 }{%
2 C_{\ell }}\} \sqrt{\frac{\pi}{2}}(3+ 6 \frac{\lambda^2 }{C_{\ell }} +\frac{%
\lambda^4 }{C_{\ell }^2} ) \frac{e^{-\frac{\lambda^2}{2 C_\ell }}}{\frac{%
\lambda}{\sqrt{C_\ell}} \sqrt{\frac \pi 2}} (1- \frac{1}{\frac{\lambda^2}{%
C_\ell}}+ \frac{3}{2 \frac{\lambda^4}{C_\ell^2}})}{ \frac{15}{4} \frac{%
\lambda^3}{C_\ell^{\frac 3 2}} \left[- \frac{\lambda }{\sqrt{C_{\ell }}}
+\exp\{\frac{\lambda^2 }{2 C_{\ell }}\} \sqrt{\frac{\pi}{2}}(1+ \frac{%
\lambda^2 }{C_{\ell }} )\frac{e^{-\frac{\lambda^2}{2 C_\ell }}}{\frac{\lambda%
}{\sqrt{C_\ell}} \sqrt{\frac \pi 2}} (1- \frac{1}{\frac{\lambda^2}{C_\ell}}+
\frac{3}{2 \frac{\lambda^4}{C_\ell^2}}) \right]^2} \\
&=&\lim_{\lambda / \sqrt{C_\ell} \to \infty} \frac{3 \frac{\lambda^5}{{C_\ell%
}^\frac 5 2} (3+5 \frac{\lambda^2}{{C_\ell}})}{\frac{15}{4} \frac{\lambda^3}{%
C_\ell^{\frac 3 2}} (3+2 \frac{\lambda^2}{{C_\ell}})^2}=1 \text{.}
\end{eqnarray*}
\end{proof}

\begin{remark}
(Some Numerical Examples) It is instructive to provide some numerical
evidence on the kurtosis of the multipole components at the North Pole, as a
function of the penalization parameter $\lambda $ and the angular power
spectrum $C_{\ell }.$ It should be recalled that $\sum_{\ell }(2\ell
+1)C_{\ell }<\infty $ by finite variance, whence $C_{\ell }$ must decay at
least as fast as $\ell ^{-2-\tau },$ some $\tau >0,$ as $\ell \rightarrow
\infty .$ For instance, considering some physically realistic values for the
power spectrum of CMB and a fixed penalization parameter $\lambda =1,$ we
have

\begin{equation*}
\begin{array}{cccccccccc}
\ell & 10 & 20 & 30 & 40 & 50 & 60 & 70 & 80 & 200 \\
C_{\ell } & 48.20 & 13.7 & 7.17 & 4.8 & 3.7 & 2.9 & 2.4 & 2.1 & 0.76 \\
\kappa_{\ell} & 3.50 & 4.08 & 4.65 & 5.19 & 5.65 & 6.21 & 6.76 & 7.22 & 15.39%
\end{array}%
\end{equation*}
\end{remark}

\subsection{Asymptotic Behavior of the Angular Trispectrum}

Exploiting the computations developed so far, it is also possible to provide
analytic expressions for the trispectra at the various multipoles, as
follows. We start recalling the results we have earlier established on the
moments of the spherical harmonic coefficients $\left\{ a_{\ell m}\right\} ,$
under the complex and real-valued regularization schemes. More precisely, we
have

\begin{itemize}
\item For all $\ell ,m$ we have%
\begin{equation}
E\{a_{\ell m}^{reg}\}=E\{a_{\ell m}^{reg\ast }\}=0\text{ ;}  \label{zero}
\end{equation}

\item For all $\ell ,m=0$
\begin{equation*}
E\left\{ (a_{\ell 0}^{reg})^{2} \right\} =E\left\{ (a_{\ell 0}^{reg\ast
})^{2}\right\}=\gamma _{0}(\frac{\lambda }{\sqrt{C_{\ell }}})\text{ ;}
\end{equation*}

\item Under the complex-valued regularization scheme, for $m\neq 0$%
\begin{equation*}
E\{\left\vert a_{\ell m}^{reg}\right\vert ^{2}\}=\gamma _{1}(\frac{\lambda }{%
\sqrt{C_{\ell }}})\text{ ,}
\end{equation*}%
while in the real-valued framework%
\begin{equation*}
E\{( a_{\ell m}^{reg {*}})^{2}\}=\gamma _{0}(\frac{\lambda }{\sqrt{C_{\ell }}%
})\text{ ;}
\end{equation*}

\item Finally for the fourth-order moments, for all ${\ell}$%
\begin{equation*}
E\left\{ (a_{\ell 0}^{reg}) ^{4}\right\}=E\{(a_{\ell {0}}^{reg\ast
})^4\}=\gamma _{2}(\frac{\lambda }{\sqrt{C_{\ell }}})
\end{equation*}%
and for $m\neq 0$%
\begin{equation*}
{E\{\left\vert a_{\ell m}^{reg }\right\vert ^{4}\}}=E\{(a_{\ell m}^{reg\ast
})^{4}\}=\gamma _{3}(\frac{\lambda }{\sqrt{C_{\ell }}})
\end{equation*}%
where
\begin{align*}
\gamma _{2}(\nu_\ell)&:=C_{\ell }^{2}\frac{4}{\sqrt{\pi }} \left\{
\sum_{k=2}^{5} (-1)^{k+1} (\frac{\nu_\ell}{\sqrt 2})^{5-k} \binom{4}{5-k}
\Gamma (\frac{k}{2};(\frac{\nu_\ell}{\sqrt 2})^2)+ 2 \sqrt \pi (\frac{%
\nu_\ell}{\sqrt 2})^{4} (1-\Phi (\frac{\nu_\ell}{\sqrt 2})) \right\}
\end{align*}
and%
\begin{align*}
\gamma _{3}(\nu_\ell)&= C_\ell^2 \sum_{k=0}^4 (-1)^k \binom{4}{k}
\nu_\ell^{4-k} \Gamma(\frac {k+2} 2; \nu_\ell^2) \text{,}
\end{align*}
\end{itemize}

where $\nu_\ell=\frac{\lambda}{\sqrt{C_\ell}}$. For {(\ref{zero})}, we note
first that it would be trivially true for an isotropic random field, but it
requires to be checked under anisotropy. In any case, the proof is
straightforward, indeed we have
\begin{eqnarray*}
E\{a_{\ell m}^{reg} \}&=&E\left\{ \rho_{\ell m}^{{\ reg}} \exp (i\psi _{\ell
m}^{{obs}})\right\} \\
&=&{\ \frac{1}{2\pi}}\int_{0}^{\infty }\int_{0}^{2\pi }r\exp (i\theta
)f_{\rho; {\ell }}(r)drd\theta \\
&=&{\ \frac{1}{2\pi}}\int_{0}^{\infty }r f_{\rho; {\ell}}(r)\left\{
\int_{0}^{2\pi }\exp (i\theta )d\theta \right\} dr=0\text{ .}
\end{eqnarray*}%
A similar argument will actually cover any product of an odd number of
spherical harmonic coefficients, because $\int_{0}^{2\pi }\exp (ik\theta
)d\theta =0$ for all non-zero integers $k$. We only need to study $%
E\{\left\vert a_{\ell m}^{reg}\right\vert ^{4}\},$ for which we have%
\begin{eqnarray*}
E\{\left\vert a_{\ell m}^{reg}\right\vert ^{4} \} &=&\int_{\lambda }^{\infty
}(r-\lambda )^{4}2\frac{r}{C_{\ell }}\exp (-\frac{r^{2}}{C_{\ell }})dr \\
&=&\int_{\lambda }^{\infty }(r-\lambda )^{4}\exp (-\frac{r^{2}}{C_{\ell }})d%
\frac{r^{2}}{C_{\ell }} \\
&=&C_{\ell }^{2}\int_{\lambda ^{2}/C_{\ell }}^{\infty }(u^{2}-4u^{3/2}\frac{%
\lambda }{\sqrt{C_{\ell }}}+6u\frac{\lambda ^{2}}{C_{\ell }}-4\sqrt{u}\frac{%
\lambda ^{3}}{\sqrt{C_{\ell }^{3}}}+\frac{\lambda ^{4}}{C_{\ell }^{2}})\exp
(-u)du \\
&=&C_{\ell }^{2}\left\{ \Gamma (3;\frac{\lambda ^{2}}{C_{\ell }})-4\frac{%
\lambda }{\sqrt{C_{\ell }}}\Gamma (\frac{5}{2};\frac{\lambda ^{2}}{C_{\ell }}%
)+6\frac{\lambda ^{2}}{C_{\ell }}\Gamma ({2};\frac{\lambda ^{2}}{C_{\ell }}%
)-4\frac{\lambda ^{3}}{\sqrt{C_{\ell }^{3}}}\Gamma ({\ \frac 3 2};\frac{%
\lambda ^{2}}{C_{\ell }})+\frac{\lambda ^{4}}{C_{\ell }^{2}}\exp (-\frac{%
\lambda ^{2}}{C_{\ell }}))\right\} \text{ ,}
\end{eqnarray*}%
using repeatedly integration by parts on the incomplete Gamma function.

\begin{remark}
It should be noted that, as expected,%
\begin{equation*}
\lim_{\lambda \rightarrow 0}E\{\left\vert a_{\ell m}^{reg}\right\vert
^{2}\}=C_{\ell }\text{ ,}
\end{equation*}%
and more generally%
\begin{equation*}
\lim_{C_{\ell }/\lambda ^{2}\rightarrow \infty }\frac{1}{C_{\ell }}%
E\{\left\vert a_{\ell m}^{reg}\right\vert ^{2}\}=1\text{ .}
\end{equation*}%
Moreover
\begin{equation*}
\lim_{\lambda \rightarrow 0}{E\{\left\vert a_{\ell m}^{reg}\right\vert ^{4}\}%
}/{C_{\ell }^{2}}=2\text{ ,}
\end{equation*}%
again as expected, because in the limiting Gaussian case%
\begin{align*}
E\{\left\vert a_{\ell m}\right\vert ^{4}\}&=E\{[{\func{Re}}(a_{lm})^{2}+{%
\func{Im}}(a_{lm})^{2}]^{2}\} \\
&=E\{{\func{Re}}(a_{lm})^{4}\}+E\{{\func{Im}}(a_{lm})^{4}\}+2E\{{\func{Re}}%
(a_{lm})^{2}\}E\{{\func{Im}}(a_{lm})^{2}\} \\
&=\frac{3}{4}C_{\ell }^{2}+\frac{3}{4}C_{\ell }^{2}+\frac{2}{4}C_{\ell
}^{2}=2C_{\ell }^{2}\text{ .}
\end{align*}
\end{remark}

The previous result can be generalized as follows.

\begin{proposition}
For all $p=1,2,3,...$ and for $m \ne 0$, we have%
\begin{equation*}
E\left\{ \left\vert a_{\ell m}^{reg}\right\vert ^{2p}\right\} =C_{\ell
}^{p}\sum_{k=0}^{2p}(-1)^{k} \binom{2p}{k} \nu_\ell ^{{2 p-k }}\Gamma (\frac{%
{k+2}}{2}; \nu_\ell^2)\text{ ,}
\end{equation*}%
where $\binom{2p}{k}=\frac{(2p)!}{(2p-k)!k!}$ is the standard binomial
coefficient.
\end{proposition}

\begin{proof}
The proof is identical to the previous arguments, and hence it is not
repeated for brevity's sake.
\end{proof}

An important consequence of these results is the following

\begin{lemma}
We have that%
\begin{equation*}
\lim_{C_{\ell }\rightarrow 0} \exp\left\{-\frac{\lambda^2}{C_\ell}\right\}
\frac{E\{\left\vert a_{\ell m}^{reg}\right\vert ^{4}\}}{\left[ E\{\left\vert
a_{\ell m}^{reg}\right\vert ^{2}\}\right] ^{2}} = 6\text{ ,}
\end{equation*}%
whence $\frac{E\{\left\vert a_{\ell m}^{reg}\right\vert ^{4}\}}{\left[
E\{\left\vert a_{\ell m}^{reg}\right\vert ^{2}\}\right] ^{2}}$ diverges
superexponentially.
\end{lemma}

\begin{proof}
It suffices to notice that
\begin{eqnarray*}
\frac{E\{\left\vert a_{\ell m}^{reg}\right\vert ^{4}\}}{\left[ E\{\left\vert
a_{\ell m}^{reg}\right\vert ^{2}\}\right] ^{2}}&=&\frac{C_{\ell }^{2}
\int_{\lambda }^{\infty } (\frac{r-\lambda }{\sqrt{C_{\ell }}})^{4} {\ \frac{%
2r}{\sqrt{C_\ell}}} \exp(-\frac{r^{2}}{C_{\ell }})d {\ \frac{r}{\sqrt{
C_{\ell }}}}}{\left\{ C_{\ell }\int_{\lambda }^{\infty }(\frac{r-\lambda }{%
\sqrt{C_{\ell }}})^{2}2\frac{r}{\sqrt{C_{\ell }}}\exp (-\frac{r^{2}}{C_{\ell
}})d\frac{r}{\sqrt{C_{\ell }}}\right\} ^{2}} \\
&=&\frac{\int_{\lambda /\sqrt{C_{\ell }}}^{\infty }(u-\frac{\lambda }{\sqrt{%
C_{\ell }}})^{4} {2u} \exp (-u^{2})du}{\left\{ 2\int_{\lambda /\sqrt{C_{\ell
}}}^{\infty }(u-\frac{\lambda }{\sqrt{C_{\ell }}})^{2}u\exp
(-u^{2})du\right\} ^{2}} \\
&=&\frac{\int_{0}^{\infty }v^{4} {\ 2 (v+\frac{\lambda}{\sqrt{C_\ell}})}\exp
(-(v+\frac{\lambda }{\sqrt{C_{\ell }}})^{2})dv}{\left\{ 2\int_{0}^{\infty
}v^{2}(v+\frac{\lambda }{\sqrt{C_{\ell }}})\exp (-(v+\frac{\lambda }{\sqrt{%
C_{\ell }}})^{2})d{v}\right\} ^{2}} \\
&=&\exp (\frac{\lambda ^{2}}{C_{\ell }}) \frac{\int_{0}^{\infty }v^{4} {(v+%
\frac{\lambda }{\sqrt{C_{\ell }}}) } \exp (-v^{2}-2\frac{\lambda v}{\sqrt{%
C_{\ell }}})dv}{{2}\left\{ \int_{0}^{\infty }v^{2}(v+\frac{\lambda }{\sqrt{%
C_{\ell }}})\exp (-v^{2}-2\frac{\lambda v}{\sqrt{C_{\ell }}})d {v} \right\}
^{2}}\text{ ,}
\end{eqnarray*}
where, by applying the expansion of the $\text{Erfc}$ function $\frac{%
e^{-x^2}}{x \sqrt{\pi}} (1- \frac{1}{2 x^2}+ \frac{3}{4 x^4}-\frac{15}{8 x^6}%
+O(x^{-7}))$, for large $x$, we have
\begin{eqnarray*}
&&\lim_{\lambda / \sqrt{C_\ell} \to \infty}\frac{\int_{0}^{\infty }v^{4} {(v+%
\frac{\lambda }{\sqrt{C_{\ell }}}) } \exp (-v^{2}-2\frac{\lambda v}{\sqrt{%
C_{\ell }}})dv}{{2}\left[ \int_{0}^{\infty }v^{2}(v+\frac{\lambda }{\sqrt{%
C_{\ell }}})\exp (-v^{2}-2\frac{\lambda v}{\sqrt{C_{\ell }}})d {v} \right]
^{2}} \\
&=&\lim_{\lambda / \sqrt{C_\ell} \to \infty} \frac{1+\frac{\lambda^2 }{{%
C_{\ell }}} -\frac {\sqrt{\pi}}{2} \frac{\lambda }{\sqrt{C_{\ell }}} \exp\{%
\frac{\lambda^2}{{C_{\ell }}}\} (3+2 \frac{\lambda^2}{{C_{\ell }}}) \text{%
Erfc}(\frac{\lambda }{\sqrt{C_{\ell }}})}{2 \left[ \frac{1}{2} (1-\exp\{%
\frac{\lambda^2 }{{C_{\ell }}}\} \sqrt{\pi} \frac{\lambda }{\sqrt{C_{\ell }}}
erfc(\frac{\lambda }{\sqrt{C_{\ell }}}))\right]^2} \\
&=&\lim_{\lambda / \sqrt{C_\ell} \to \infty} \frac{1+\frac{\lambda^2 }{{%
C_{\ell }}} -\frac {\sqrt{\pi}}{2} \frac{\lambda }{\sqrt{C_{\ell }}} \exp\{%
\frac{\lambda^2}{{C_{\ell }}}\} (3+2 \frac{\lambda^2}{{C_{\ell }}}) \frac{%
e^{-\frac{\lambda^2}{ C_\ell }}}{\frac{\lambda}{\sqrt{C_\ell}} \sqrt{\pi}}
(1- \frac{1}{2 \frac{\lambda^2}{C_\ell}}+ \frac{3}{4 \frac{\lambda^4}{%
C_\ell^2}}-\frac{15}{8 x^6})}{2 \left[ \frac{1}{2} (1-\exp\{\frac{\lambda^2
}{{C_{\ell }}}\} \sqrt{\pi} \frac{\lambda }{\sqrt{C_{\ell }}} \frac{e^{-%
\frac{\lambda^2}{ C_\ell }}}{\frac{\lambda}{\sqrt{C_\ell}} \sqrt{\pi}} (1-
\frac{1}{2 \frac{\lambda^2}{C_\ell}}+ \frac{3}{4 \frac{\lambda^4}{C_\ell^2}}-%
\frac{15}{8 x^6}))\right]^2} \\
&=&\lim_{\lambda / \sqrt{C_\ell} \to \infty} \frac{24 \frac{\lambda^6}{{%
C_\ell}^3}(15+4 \frac{\lambda^2}{{C_\ell}})}{(15-6 \frac{\lambda^2}{{C_\ell}}%
+4 \frac{\lambda^4}{{C_\ell}^2})^2}=6 \text{ . }
\end{eqnarray*}
\end{proof}

In view of the previous results, it is simple to provide exact analytic
expressions for the expected trispectra $E\{T_{\ell }\}$ under both
regularization schemes, and to study their asymptotic behavior as the
frequencies increase. We obtain

\begin{theorem}
We have%
\begin{align*}
E\left\{ T_{\ell }^{reg}(\theta ,\phi ) ^{4}\right\}&=\gamma _{2}(\frac{%
\lambda }{\sqrt{C_{\ell }}})\left\vert Y_{\ell 0}(\theta ,\phi )\right\vert
^{4}+\gamma _{3}(\frac{\lambda }{\sqrt{C_{\ell }}})\sum_{m\neq 0}\left\vert
Y_{\ell m}(\theta ,\phi )\right\vert ^{4} \\
&\:\:+{2}\gamma _{0}(\frac{\lambda }{\sqrt{C_{\ell }}})\gamma _{1}(\frac{%
\lambda }{\sqrt{C_{\ell }}})\left\vert Y_{\ell 0}(\theta ,\phi )\right\vert
^{2}\left\{ \frac{2\ell +1}{4\pi }-\left\vert Y_{\ell 0}(\theta ,\phi
)\right\vert ^{2}\right\} \\
&\:\:+\gamma _{1}^{2}(\frac{\lambda }{\sqrt{C_{\ell }}}){\sum_{m^{\prime
}\neq m}, {\ m,m^{\prime }\neq 0}}\left\vert Y_{\ell m}(\theta ,\phi
)\right\vert ^{2}\left\vert Y_{\ell m^{\prime }}(\theta ,\phi )\right\vert
^{2}.
\end{align*}%
Likewise%
\begin{align*}
E\left\{ T_{\ell }^{reg\ast }(\theta ,\phi )\right\} ^{4}&=\gamma _{2}(\frac{%
\lambda }{\sqrt{C_{\ell }}})\sum_{m}\left\vert Y_{\ell m}^{\mathcal{R}%
}(\theta ,\phi )\right\vert ^{4}+\gamma _{0}^{2}(\frac{\lambda }{\sqrt{%
C_{\ell }}})\sum_{m^{\prime }\neq m}\left\vert Y_{\ell m}(\theta ,\phi
)\right\vert ^{2}\left\vert Y_{\ell m^{\prime }}(\theta ,\phi )\right\vert
^{2}.
\end{align*}%
As $\lambda /\sqrt{C_{\ell }}\rightarrow \infty $, the trispectrum is then
asymptotic to
\begin{equation*}
\lim_{\lambda /\sqrt{C_{\ell }}\rightarrow \infty }\frac{E\left\{ T_{\ell
}^{reg}(\theta ,\phi )^{4}\right\} }{E\left\{ T_{\ell
}^{reg}(0,0)^{4}\right\} P_{\ell }^{4}(\cos \theta )}=1\text{ .}
\end{equation*}%
For the real-valued regularization scheme we get%
\begin{equation*}
\lim_{\lambda /\sqrt{C_{\ell }}\rightarrow \infty }\frac{E\left\{ T_{\ell
}^{reg\ast }(\theta ,\phi )^{4}\right\} }{E\{T_{\ell }^{reg\ast
}(0,0)^{4}\}V_{\ell }(\theta ,\phi )}=1\text{,}
\end{equation*}%
where as $\ell \to \infty$
\begin{equation*}
V_{\ell }(\theta ,\phi )=\left( \frac{4\pi }{2\ell +1}\right)
^{2}\sum_{m}(Y_{\ell ,m}^{\mathcal{R}})^{4}\rightarrow
\begin{cases}
1, & \text{for }(\theta ,\phi )=(0,0), \\
0\text{ a.e.}, & \text{otherwise }.%
\end{cases}%
\end{equation*}
\end{theorem}

\begin{proof}
Recall that $E\left\{ T_{\ell }^{reg}(0,0)^{4}\right\} =E\{\left\vert
a_{\ell 0}^{reg}\right\vert ^{4}\}\left\{ \frac{2\ell +1}{4\pi }\right\} ^{2}
$, and note that%
\begin{align}
&E\left\{ T_{\ell }^{reg}(\theta ,\phi )^{4}\right\}  \notag \\
&=\sum_{m}E\{\left\vert a_{\ell m}^{reg}\right\vert ^{4}\}\left\vert Y_{\ell
m}(\theta ,\phi )\right\vert ^{4}+\sum_{m\neq m^{\prime }}E\{\left\vert
a_{\ell m}^{reg}\right\vert ^{2}\}E\{\left\vert a_{\ell m^{\prime
}}^{reg}\right\vert ^{2}\}\left\vert Y_{\ell m}(\theta ,\phi )\right\vert
^{2}\left\vert Y_{\ell {m^{\prime }}}(\theta ,\phi )\right\vert ^{2}  \notag
\\
&=\sum_{m\neq 0}E\{\left\vert a_{\ell m}^{reg}\right\vert ^{4}\}\left\vert
Y_{\ell m}(\theta ,\phi )\right\vert ^{4} +\sum_{m\neq m^{\prime
}}E\{\left\vert a_{\ell m}^{reg}\right\vert ^{2}\}E\{\left\vert a_{\ell
m^{\prime }}^{reg}\right\vert ^{2}\}\left\vert Y_{\ell m}(\theta ,\phi
)\right\vert ^{2}\left\vert Y_{\ell {m^{\prime }}}(\theta ,\phi )\right\vert
^{2}  \label{ghione2} \\
&\;\;+ E\{( a_{\ell 0}^{reg})^{4}\}\left\{ \frac{2\ell +1}{4\pi }\right\}
^{2}P_{\ell }^{4}(\cos \theta )  \notag
\end{align}%
whence it suffices to notice that the expected values in (\ref{ghione2}) are
all of smaller order with respect to $E\{( a_{\ell 0}^{reg})^{4}\}$ . Indeed
from (\ref{ghione3}) it follows easily that%
\begin{equation*}
\lim_{{\lambda }/\sqrt{C_{\ell }}\rightarrow \infty }\frac{\sum_{m\neq
m^{\prime }}E\{\left\vert a_{\ell m}^{reg}\right\vert^{2} \} E\{\left\vert
a_{\ell m^{\prime }}^{reg}\right\vert^{2} \}}{E\left\{ (a_{\ell
0}^{reg})^4\right\} }\leq \lim_{{\lambda }/\sqrt{C_{\ell }}\rightarrow
\infty }{{\ (4 \ell^2+2 \ell) \max_{m\neq m^{\prime }}}}\frac{E\{\left\vert
a_{\ell m}^{reg}\right\vert ^{2}\}E\{\left\vert a_{\ell m^{\prime
}}^{reg}\right\vert ^{2}\}}{E\left\{ (a_{\ell 0}^{reg})^{4}\right\} }=0\text{
.}
\end{equation*}%
Note also that%
\begin{eqnarray*}
E\{\left\vert a_{\ell m}^{reg}\right\vert ^{4}\} &=&\int_{0}^{\infty
}r^{4}f_{\rho _{\ell m}}(r)dr \\
&=&\int_{\lambda }^{\infty }(r-\lambda )^{4}2\frac{r}{C_{\ell }}\exp (-\frac{%
r^{2}}{C_{\ell }})dr \\
&=&2C_{\ell }^{2}\int_{{\lambda }}^{\infty }(\frac{r-\lambda }{\sqrt{C_{\ell
}}})^{4}\frac{r}{\sqrt{C_{\ell }}}\exp (-\frac{r^{2}}{C_{\ell }})d\frac{r}{%
\sqrt{C_{\ell }}} \\
&=&2C_{\ell }^{2}\int_{\lambda /\sqrt{C_{\ell }}}^{\infty }(u-\frac{\lambda
}{\sqrt{C_{\ell }}})^{4}u\exp (-u^{2})du\text{ .}
\end{eqnarray*}%
It follows that%
\begin{eqnarray*}
\frac{E\{\left\vert a_{\ell m}^{reg}\right\vert ^{4}\}}{E\{( a_{\ell
0}^{reg})^{4}\}} &=&\frac{2C_{\ell }^{2}\int_{\lambda /\sqrt{C_{\ell }}%
}^{\infty }(u-\frac{\lambda }{\sqrt{C_{\ell }}})^{4}u\exp (-u^{2})du}{\frac{{%
2}}{\sqrt{2\pi }}C_{\ell }^{2}\int_{\lambda /\sqrt{C_{\ell }}}^{\infty }(u-%
\frac{\lambda }{\sqrt{C_{\ell }}})^{4}\exp \left\{ -\frac{u^{2}}{2}\right\}
du} \\
&\leq &K\exp \left\{ -\frac{\lambda ^{2}(1-\varepsilon )}{2C_{\ell }}%
\right\} ,\text{ any }\varepsilon >0\text{ ,}
\end{eqnarray*}%
so that
\begin{equation*}
\lim_{{\lambda }/\sqrt{C_{\ell }}\rightarrow \infty }\frac{E\{\left\vert
a_{\ell m}^{reg}\right\vert ^{4}\}}{E\left\{ (a_{\ell 0}^{reg})^{4}\right\} }%
=0\text{ }.
\end{equation*}%
It is then immediate to see that $\left[ (\ref{ghione2})/E\left\{ T_{\ell
}^{reg}(0,0)^{4}\right\} \right] \rightarrow 1$ as $\lambda /\sqrt{C_{\ell }}%
\rightarrow \infty ,$ whence our first result is established. By an
analogous argument, it is easy to see that%
\begin{align*}
\frac{E\left\{ T_{\ell }^{reg *}(\theta ,\phi )^{4}\right\} }{E\{T_{\ell
}^{reg\ast }(0,0)^{4}\}V_{\ell }(\theta ,\phi )}&=\frac{\sum_{m}E\{\left%
\vert a_{\ell m}^{reg*}\right\vert ^{4}\}\left\vert Y_{\ell m}^{\mathcal{R}%
}(\theta ,\phi )\right\vert ^{4}}{E\{T_{\ell }^{reg\ast }(0,0)^{4}\}V_{\ell
}(\theta ,\phi )} \\
&\;\;+\frac{\sum_{m\neq m^{\prime }}E\{\left\vert a_{\ell
m}^{reg*}\right\vert ^{2}\}E\{\left\vert a_{\ell m^{\prime
}}^{reg*}\right\vert ^{2}\}\left\vert Y_{\ell m}^{\mathcal{R}}(\theta ,\phi
)\right\vert ^{2}\left\vert Y_{\ell m^{\prime }}^{\mathcal{R}}(\theta ,\phi
)\right\vert ^{2}}{E\{T_{\ell }^{reg\ast }(0,0)^{4}\}V_{\ell }(\theta ,\phi )%
}\rightarrow 1\text{ ,}
\end{align*}%
so that we need only investigate the asymptotic behavior of
\begin{equation*}
V_{\ell }(\theta ,\phi ):=\left\{ \frac{4\pi }{2\ell +1}\right\}
^{2}\sum_{m=-\ell }^{\ell }(Y_{\ell ,m}^{\mathcal{R}})^{4}\text{ .}
\end{equation*}%
To this aim, we recall the following recent result by Sogge and Zelditch
\cite{sogge}; as $\ell \rightarrow \infty $%
\begin{equation*}
\frac{1}{2\ell +1}\int_{S^{2}}\sum_{m=-\ell }^{\ell }\left\vert Y_{\ell
,m}(x)\right\vert ^{4}dx=o(\left\{ \log \ell \right\} ^{1/4})\text{ .}
\end{equation*}%
Now of course%
\begin{eqnarray*}
\sum_{m=-\ell }^{\ell }\left\vert Y_{\ell ,m}(x)\right\vert ^{4}
&=&\sum_{m=-\ell }^{\ell }\left\{ \left\vert \func{Re}(Y_{\ell
,m}(x))\right\vert ^{2}+\left\vert \func{Im}(Y_{\ell ,m}(x))\right\vert
^{2}\right\} ^{2} \\
&=&\sum_{m=-\ell }^{\ell }\left\{ \left\vert \func{Re}(Y_{\ell
,m}(x))\right\vert ^{4}+\left\vert \func{Im}(Y_{\ell ,m}(x))\right\vert
^{4}+2\left\vert \func{Re}(Y_{\ell ,m}(x))\right\vert ^{2}\left\vert \func{Im%
}(Y_{\ell ,m}(x))\right\vert ^{2}\right\} \\
&\geq &\frac{1}{2}\sum_{m=-\ell }^{\ell }(Y_{\ell ,m}^{\mathcal{R}})^{4}%
\text{ ,}
\end{eqnarray*}%
from which we obtain immediately%
\begin{equation*}
\left\{ \frac{4\pi }{2\ell +1}\right\} ^{2}\sum_{m=-\ell }^{\ell }(Y_{\ell
,m}^{\mathcal{R}})^{4}=o(\frac{\left\{ \log \ell \right\} ^{1/4}}{\ell })%
\text{ for almost all }x\in S^{2},
\end{equation*}%
as claimed.
\end{proof}

\begin{remark}
The previous Theorem can be expressed in plain words as follows: under the
complex-valued regularization scheme, after normalization, the trispectrum
behaves asymptotically as the fourth power of the Legendre polynomial. In
the real-valued case, the normalized trispectrum behaves as the averaged sum
of the fourth-powers of (real-valued) spherical harmonics. The first result
is heuristically explained considering that sparsity will enforce the choice
of the single coefficient $a_{\ell 0}^{reg}$ more and more often,as $\lambda
/\sqrt{C_{\ell }}\rightarrow 0;$ in the latter case, each coefficient $%
a_{\ell m}^{\mathcal{R}}$ has the same probability to be selected, as they
are all identically distributed: however in the limit at most one of them
will be nonzero, so the trispectrum will reproduce the oscillations of a
single (randomly chosen) functions $Y_{\ell m}^{\mathcal{R}}$. Note that as $%
\ell \rightarrow \infty ,$ $P_{\ell }(\cos \theta )\rightarrow 0$ for all $%
\theta \neq 0,\pi ,$ whence in both cases the trispectrum at the Poles has a
dominating behavior with respect to almost all other directions.
\end{remark}

\section{Some Concluding Remarks \label{conclusions}}

In this paper, we have shown that convex regularization of spherical
isotropic Gaussian fields with a Fourier dictionary does not preserve in
general the Gaussianity and isotropy properties of the input random fields.
We refer to \cite{feeney} for more discussions form a physical point of view
and ample numerical evidence to illustrate these claims, in a setting
related to Cosmological data analysis.

In a nutshell, our arguments can be summarized as follows. The
result of convex regularization is basically a form of
soft-thresholding on the spherical harmonic coefficients $\left\{
a_{\ell m}\right\}.$ These coefficients are hence independent and
nonGaussian, whence anisotropy follows. Indeed, this finding
complements earlier results from \cite{BaMa, BMV, MarPecBook, BT},
entailing that independent coefficients in a spherical harmonic
basis are necessarily Gaussian under isotropy, and therefore
cannot be sparse in the usual meaning with which this concept is
understood.

It seems hence quite natural to extend our results and to suggest that for
Gaussian isotropic random fields defined on homogeneous spaces of
noncommutative groups, sparsity cannot be imposed on the random coefficients
of a Fourier basis. The crucial difficulty here is the choice of a Fourier
basis as a sparsity dictionary in a noncommutative setting; in particular it
should be noted that our arguments do not entail that anisotropy will arise
when choosing, for instance, a wavelet frame as a dictionary. Likewise, no
anisotropy would arise for homogeneous spaces of commutative groups: for
instance, soft- or hard- thresholding the random Fourier coefficients of
isotropic random fields on the circle does not make these random fields
anisotropic. It is indeed the noncommutative manifold structure of the
sphere, and in particular the multiplicity of eigenfunctions corresponding
to the same eigenvalue, which brings in a conflict between independence and
nonGaussianity, under isotropy assumptions; because the random spherical
harmonics coefficients arising from convex regularization are independent
and nonGaussian, anisotropy follows.

The paper has presented a number of characterizations of such anisotropy as
a function of the angular power spectra of the input fields; the relevance
of these findings for the areas where convex regularization is exploited as
a preliminary step for spherical data analysis is going to be investigated
elsewehere.

\end{document}